\documentclass[a4paper,reqno,11pt]{amsart}
\usepackage[francais,english]{babel}
\usepackage{amssymb,url,xspace}

\usepackage[all]{xy}
\usepackage[latin1]{inputenc}
\usepackage{ae,aecompl}
\usepackage{amscd}
\usepackage{array}
\usepackage{graphicx}
\graphicspath{{FIGURES_cycles/}}

\theoremstyle{plain}
\newtheorem{theoremalph}{Théorème}

\author{Guillaume Bulteau}
\address{Institut Montpelliérain Alexander Grothendieck (Imag)\\ UMR  CNRS 5149 - Université Montpellier 2\\
Case courrier 051\\
34095 Montpellier cedex 5 - France}
\email{guillaume.bulteau@math.univ-montp2.fr}
\title{Cycles géométriques réguliers}
\begin{document}

%\frontmatter

%%%%%%%%%%%%%%%%%%%%%%%%%%%%%%%%%%%%%%%%%%%%%%%%%%%%%%%%%%%%%%%%%%%%

\begin{abstract}
Soit $\pi$ un groupe de présentation finie. Pour une classe d'homologie $h$ non nulle dans $H_n(\pi;\mathbb{Z})$,  Gromov a énoncé (dans \cite{Gromov-FRM}, §6) l'existence de cycles géométriques qui représentent $h$, de volume systolique relatif aussi proche que l'on veut de celui de $h$, pour lesquels on dispose d'un contrôle sur le volume des boules dont le rayon est plus petit qu'une fraction de la systole relative du cycle. L'objectif  de cette note est d'expliquer ce résultat et d'en présenter une démonstration complète.
\end{abstract}

%\begin{altabstract}
%Let $\pi$ be a finitely presented group. If $h$ is a non trivial homology class in $H_n(\pi;\mathbb{Z})$, a theorem of Gromov (see \cite{Gromov-FRM}, §6) asserts the existence of regular geometric cycles which represent $h$, whose relative systolic volume is as close as desired to the systolic volume of $h$, in which we can control the volume of balls of radius less than half of the cycle's relative systol. The aim of this note is to explain and provide a complete proof of this result.
%\end{altabstract}

%\altkeywords{Geometric cycle, systol, systolic volume, Eilenberg-McLane space, cubical complex}
\thanks{Je remercie Florent Balacheff, Jacques Lafontaine et Stéphane Sabourau pour toutes les remarques qu'ils ont pu faire sur les premières versions de ce texte, ainsi qu'Ivan Babenko pour les nombreuses discussions sur ce sujet. Je remercie également un rapporteur anonyme pour ses commentaires, qui ont permis d'améliorer grandement certaines parties de ce texte.\\
Ce travail est financé par l'ANR Finsler}
%%%%%%%%%%%%%%%%%%%%%%%%%%%%%%%%%%%%%%%%%%%%%%%%%%%%%%%%%%%%%%%%%%%%
%%%%%%%%%%%%%%%%%%%%%%%%%%%%%%%%%%%%%%%%%%%%%%%%%%%%%%%%%%%%%%%%%%%%
\maketitle

%%%%%%%%%%%%%%%%%%%%%%%%%%%%%%%%%%%%%%%%%%%%%%%%%%%%%%%%%%%%%%%%%%%%
%%%%%%%    MES COMMANDES
%%%%%%%%%%%%%%%%%%%%%%%%%%%%%%%%%%%%%%%%%%%%%%%%%%%%%%%%%%%%%%%%%%%%
\def\R{\text{\normalfont{I\hspace{-.15em}R}}}
\def\N{\text{\normalfont{I\hspace{-.15em}N}}}
\newcommand {\Z} {\mathbb{Z}}
\renewcommand{\dfrac}{\displaystyle \frac}
\newcommand{\dint}{\displaystyle \int}
\newcommand{\dsum}{\displaystyle \sum}
\newcommand{\norm}[1]{\left|\!\left| #1 \right|\!\right|}
\renewcommand{\mod}[1]{\left| #1 \right|}
\renewcommand{\le}{\leqslant}
\renewcommand{\ge}{\geqslant}
\newcommand {\eps} {\varepsilon}
\newcommand{\sys}{\mbox{\normalfont{\text{syst}}}}
\newcommand {\vol}{\mbox{\normalfont{\text{vol}}}}
\newcommand {\Vol}{\mbox{\normalfont{\text{Vol}}}}
\newcommand {\mass}{\mbox{\normalfont{\text{mass}}}}
\newcommand {\Long}{\mbox{\normalfont{long}}}
\newcommand{\dist}{\normalfont{\text{dist}}}
\newcommand {\volremp}{\normalfont{\text{Vol\:Remp}}}
\newcommand {\id}{\normalfont{\text{Id}}}
\newcommand{\quot}[2]{\raise0.7ex\hbox{$#1$} \!\mathord{\left/
 {\vphantom { {n}}}\right.\kern-\nulldelimiterspace}
\!\lower0.7ex\hbox{$#2$}}
\newcommand {\K} {K(\pi,1)}
\newcommand{\tub}{\normalfont{\text{Tube}}}
\renewcommand{\d}{\normalfont{\text{d}}}
\newcommand{\eq}{\begin{eqnarray*}}
\newcommand{\fineq}{\end{eqnarray*}}
\newcommand{\w}{\widetilde}
\newcommand \set[1]{ \left\{#1\right\}}
\newcommand {\im}{\mbox{\normalfont{Im}}}
\newcommand{\et}{\ \ \text{et}\ \ }
\newcommand{\tend}{\underset{+\infty}{\longrightarrow}}
%%%%%%%%%%%%%%%%%%%%%%%%%%%%%%%%%%%%%%%%%%%%%%%%%%%%%%%%%%%%%%%%%%%%%
%%%%%%%%%%%%%%%%%%%%%%%%%%%%%%%%%%%%%%%%%%%%%%%%%%%%%%%%%%%%%%%%%%%%%
%%%%%%%%%%%%%%%%%%%%%%%%%%%%%%%%%
\theoremstyle{plain}
\newtheorem{theo}{Théorème}
\newtheorem{lemm}{Lemme}
\newtheorem{résultat}{Résultat}
\newtheorem{assertion}{Assertion}
\newtheorem{question}{Question}
\newtheorem{com}{Commentaire}
\newtheorem*{gromov}{Inégalité systolique de Gromov}
\newtheorem*{inegiso}{Inégalité isopérimétrique dans \boldmath $L$\unboldmath}
%%%%%%%%%%%%%%%%%%%%%%%%%%%%%%%%
\theoremstyle{remark}
\newtheorem{rema}{Remarque}
\newtheorem{Rema-ess}[rema]{Remarque essentielle}
\theoremstyle{definition}
\newtheorem{defi}{Définition}
\newtheorem{exem}{Exemple}
%%%%%%%%%%%%%%%%%%%%%%%%%%%%%%%%
%%%%%%%%%%%%%%%%%%%%%%%%%%%%%%%%

%%%%%%%%%%%%%%%%%%%%%%%%%%%%%%%%%%%%%%%%%%%%%%%%%%%%%%%%%%%%%%%%%%%%
%%%%%%%%%%%%%%%%%%%%%%%%%%%%%%%%%%%%%%%%%%%%%%%%%%%%%%%%%%%%%%%%%%%%
%\mainmatter

%%%%%%%%%%%%%%%%%%%%%%%%%%%%%%%%%%%%%%%%%%%%%%%%%%%%%%%%%%%%%%%%%%%%
%%%%%%%%%%%%%%%%%%%%%%%%%%%%%%%%%%%%%%%%%%%%%%%%%%%%%%%%%%%%%%%%%%%%
\renewcommand\contentsname{Sommaire}
\tableofcontents
%%%%%%%%%%%%%%%%%%%%%%%%%%%%%%%%%%%%%%%%%%%%%%%%%%%%%%%%%%%%%%%%%%%%
%%%%%%%%%%%%%%%%%%%%%%%%%%%%%%%%%%%%%%%%%%%%%%%%%%%%%%%%%%%%%%%%%%%%
\section{Introduction}
%%%%%%%%%%%%%%%%%%%%%%%%%%%%%%%%%%%%%%%%%%%%%%%%%%%%%%%%%%%%%%%%%%%%%%%%%%%%%%%%%%%

%%%%%%%%%%%%%%%%%%%%%%%%%%%%%%%%%%%%%%%%%%%%%%%%%%%%%%%%%%%%%%%%%%%%%%%%%%%%%%%%%%%
%%%%%%%%%%%%%%%%%%%%%%%%%%%%%%%%%%%%%%%%%%%%%%%%%%%%%%%%%%%%%%%%%%%%%%%%%%%%%%%%%%%
\subsection{Le cadre et le but}
Dans la suite $\pi$ désignera un groupe de présentation finie et $n\ge 2$ est un entier. Considérons l'espace d'{Eilenberg-MacLane} $\K$ et prenons une classe d'homologie $h$ dans $H_n(\pi;Z)$ qui, par définition, est le groupe d'homologie $H_n(\K;\Z)$. Rappelons que $\K$ est un CW-complexe connexe tel que $\pi_1(\K)=\pi$ et $\pi_i(\K)=0$ pour tout $i\ge 2$ entier. Un tel espace est unique à type d'homotopie près (voir \cite{Hatcher}). \\

On appellera \textit{cycle géométrique représentant $h$} un triplet $(V,f,g)$ où 
\begin{quote}
\begin{itemize}
	\item $V$ est une pseudo-variété orientable de dimension $n$;
	\item $f:V\to \K$ est un application continue telle que  $f_*[V]=h$, 
$[V]$ étant la classe fondamentale de $V$ et $$f_*:H_n(V;\Z)\to H_n(\pi;\Z)$$ étant l'application induite par $f$ au niveau des groupes d'homologie;
\item $g$ une métrique riemannienne lisse par morceaux sur $V$.\\
\end{itemize}
\end{quote}

Les définitions précises des objets qui interviennent ci-dessus sont rappelées au paragraphe \ref{vocabulaire}. Une représentation $(V,f,g)$ de $h$ sera dite \textit{normale} lorsque $f$ induit au niveau des groupes fondamentaux un épimorphisme. On peut noter que toute classe entière est représentable par un cycle géométrique et cette représentation peut être normalisée (voir \cite{Babenko-TSU}). \\

Soient $(V,f,g)$ un cycle géométrique représentant $h$ et $v\in V$. On note $\sys_v(V,f,g)$ la longueur du plus petit lacet $c$ de point initial $v$ dans $V$ tel que le lacet $f\circ c$ soit non contractile dans $\K$. On définit encore \textit{la systole relative} de $(V,f,g)$ par :
$$\sys(V,f,g)=\underset{v\in V}{\inf}\:\sys_v(V,f,g)$$

Le \textit{volume systolique relatif} de $(V,f,g)$ est alors : $$\sigma(V,f,g)=\dfrac{\vol(V,g)}{\sys(V,f,g)^n}.$$

Cela permet de définir \textit{le volume systolique de $h$} notée $\sigma(h)$, qui est l'infimum des $\sigma(V,f,g)$ lorsque $(V,f,g)$ décrit la collection des cycles géo\-mé\-triques qui représentent $h$. 

Lorsque $h\neq 0$, selon Gromov (voir \cite{Gromov-FRM}, § 6), on a $\sigma(h)>0$. Cependant on ne sait pas si cet infimum est réalisé, ni quelle est la structure des pseudo-variétés qui le réalisent éventuellement. Dans le cas où la classe $h$ est réalisable par une variété, on sait qu'il coïncide avec le volume systolique de n'importe quelle représentation normale par une variété de $h$, voir \cite{Babenko-TSU}, \cite{Babenko-ATSU} et \cite{Brunnbauer-HI}. Rappelons que le volume systolique d'une variété compacte $M$ de dimension $n$ est donné par :
$$\sigma(M)=\underset{g}{\inf}\dfrac{\vol(M,g)}{\sys(M,g)^n},$$
où l'infimum est pris sur toutes les métriques riemanniennes sur $M$ et où $\sys(M,g)$ désigne la longueur du plus petit lacet non contractile dans $(M,g)$. 
On sait aussi, d'après un résultat de Babenko et Balacheff (voir \cite{Babenko-Balacheff-DVS}), que toute classe entière $h$ dans $H_n(\pi;\Z)$ admet une représentation normale par une pseudo-variété admissible $V$ (i.e. telle que tout élément de $\pi_1(V)$ peut être représenté par un lacet qui ne rencontre pas le lieu singulier de $V$) et qu'alors le volume systolique de $h$ est donné par :
$$\sigma(h)=\underset{g}{\inf}\dfrac{\vol(V,g)}{\sys(V,f,g)^n},$$
où l'infimum est pris sur toutes les métriques polyédrales sur $V$.\\
Pour un panorama de la géométrie systolique, on peut consulter \cite{Katz-SG}.\\

Le but de ce texte  est de proposer une démonstration détaillée du résultat suivant, dû à {Gromov} (voir \cite{Gromov-FRM} page 71), dont la preuve existante est très lacunaire.

%%%%%%%%%%%%%%%%%%%%%%%%%%%%%%%%%%%%%%%%%%%%%%%%%%%%%%%%%%%%%%%%%%%%%%%%%%%%%%%%%%%
\begin{theoremalph}
\label{theoregularisation}
Soient $\pi$ un groupe de présentation finie et $h$ une classe d'homologie entière non nulle dans $H_n(\pi;\Z)$. Pour tout $\eps\in]0,\tfrac12\sys(V,f,g)[$, il existe un cycle géométrique $(V,f,g)$ représentant $h$ tel que :
\begin{enumerate}
	\item $\sigma(V,f,g)\le\sigma(h)+\eps$.
	\item Pour $R\in [\eps,\tfrac 12\:\sys(V,f,g)]$, les boules $B(R)$ de rayon $R$ dans  $V$ vérifient :
\begin{equation}
\label{eqvolboule}
		\vol (B(R))\ge A_n R^n
\end{equation}
pour une certaine constante universelle $A_n$, qui ne dépend que de la dimension de $h$.
\end{enumerate}
Un tel cycle géométrique est dit $\eps$-régulier\index{Cycle géométrique régulier}.
\end{theoremalph}

Pour illustrer de manière élémentaire ce théorème A, donnons une idée de la manière dont l'inégalité (\ref{eqvolboule}) sur le volume des boules permet de préciser la topologie des cycles géométriques $\eps$-réguliers. 

Soit $(V,f,g)$ un cycle $\eps$-régulier qui représente une classe d'homologie non triviale dans $H_n(\pi;\Z)$. Considérons un système maximal $B_1,\ldots,B_N$ de boules ouvertes disjointes de $V$ de rayon $R_0=\tfrac 1{2}\:\sys(V,f,g)$. Les boules concentriques $2B_1,\ldots,2B_N$ de rayon $2R_0$ recouvrent $V$. Appelons $\mathcal N$ le nerf de ce recouvrement\index{Nerf d'un recouvrement}. Il s'agit du complexe simplicial construit de la manière suivante :

\begin{quote}
\begin{itemize}
	\item Les sommets $p_1,\ldots,p_N$ de $\mathcal N$ sont identifiés avec les boules du recouvrement ;
	\item Deux sommets $p_i$ et $p_j$ sont reliés par une arête lorsque $2B_i\cap 2B_j\neq \emptyset$;
	\item Les sommets $p_{i_0},\ldots ,p_{i_p}$ forment un simplexe de dimension $p$ de $\mathcal N$ lorsque :
	$$2B_{i_0}\cap\cdots\cap 2B_{i_p}\neq \emptyset.$$
\end{itemize}
\end{quote}

On peut alors borner le nombre $N_k$ de simplexes de dimension $k$ de $\mathcal N$ en fonction du volume systolique $\sigma(h)$ de la classe $h$. Par exemple, on a $$\vol(V,g)\ge \dsum_{i=1}^{N_0}\vol (B_i)\ge N_0 A_n R_0^n,$$
ce qui permet de borner le nombre de sommets $N_0$ de $\mathcal N$ en fonction de $\sigma(h)$.\\

Gromov utilise le théorème A afin de relier des propriétés topologiques de $h$ au volume systolique. Plus précisément, ces cycles réguliers permettent à Gromov d'obtenir des inégalités entre le volume systolique et deux invariants topologiques importants de la classe $h$, qui sont :

\begin{itemize}
	\item La hauteur simpliciale $h_s(h)$ de $h\in H_m(\pi;\Z)$, qui est le nombre minimal de simplexes de toute dimension d'un cycle géométrique qui représente $h$;
	\item Le volume simplicial\index{Volume simplicial} $\norm{h}_\Delta$, défini comme l'infimum des sommes $\dsum \mod{r_i}$ sur toutes les représentations de $h$ par des cycles singuliers réels $\dsum r_i\sigma_i$.\\
	\end{itemize}

Gromov a notamment obtenu les résultats suivants (voir \cite{Gromov-FRM}, théorème 6.4.C'' et théorème 6.4.D' et \cite{Gromov-Actes}, paragraphe 3.C.3).\\

\begin{theo}[Gromov]
\label{theoGromovIllus}
Soit $\pi$ un groupe de présentation finie, $h\in H_m(\pi;\Z)$, une classe d'homologie non nulle de dimension $m\ge 2$.
\begin{enumerate}
	\item Il existe deux constantes positives $C_m$ et $C_m'$, qui ne dépendent que de $m$, telles que :
	$$\sigma(h)\ge C_m\dfrac{h_s(h)}{\exp(C_m'\sqrt {\ln h_s(h)})}.$$
\item	Il existe une constante positive $C_m''$ qui ne dépend que de la dimension $m$ telle que :
$$\sigma(h)\ge C_m''\dfrac{\norm{h}_\Delta}{(\ln (2+\norm{h}_\Delta))^m}.$$
\end{enumerate}
\end{theo}

Le théorème A a notamment encore  été utilisé dans \cite{Sabourau-SVME}, page 168, afin de déterminer une borne supérieure sur la $H$-entropie de ces cycles. Avant de donner le résultat obtenu, précisons la définition de la $H$-entropie d'un cycle géométrique $(V,f,g)$ qui représente une classe d'homologie non triviale dans $H_n(\pi;\Z)$. L'application $f$ induit un morphisme $f_*:\pi_1(V)\to \pi$ au niveau des groupes fondamentaux. On note $H$ le noyau de $f_*$ et on considère le revêtement galoisien $\w V\to V$ associé au sous-groupe distingué $H$ de $\pi_1(V)$. Notons encore  $\w g$ la métrique induite par celle de $V$ sur  $\w V$ et fixons $v_0$ dans $V$.  Soit $\w v_0$ un relevé de $v_0$ dans $\w V$. L'entropie volumique de $(V,f,g)$ associée à $H$ (ou $H$-entropie) est alors :
$$\text{Ent}_H(V,f,g)=\underset{R\to +\infty}{\lim}\dfrac1R\:\log \vol_{\widetilde g}(B(\w v_0,R)).$$
Sabourau a  établi le résultat suivant.\\

\begin{theo}[Sabourau]
\label{theoSabourauIllus}
Pour tout cycle géométrique $(V,f,g)$ $\eps$-régulier qui représente une classe d'homologie non triviale dans $H_n(\pi;\Z)$, on a :
$$\text{\normalfont{Ent}}_H(V,g)\le \dfrac{1}{\beta\: \sys(V,f,g)}\log\left(\dfrac{\sigma(V,f,g)}{A_n\alpha^n}\right)$$
où $\alpha\ge \eps$, $\beta>0$, $4\alpha+\beta<\tfrac 12$ et la constante $A_n$ est donnée par le théorème \normalfont{A}.
\end{theo}

%%%%%%%%%%%%%%%%%%%%%%%%%%%%%%%%%%%%%%%%%%%%%%%%%%%%%%%%%%%%%%%%%%%%%%%%%%%%%%%%%%

\subsection{Idée générale de la démonstration}
%%%%%%%%%%%%%%%%%%%%%%%%%%%%%%%%%%%%%%%%%%%%%%%%%%%%%%%%%%%%%%%%%%%%%%%%%%%%%%%%%%%

Pour démontrer le théorème A, on suivra la procédure donnée par Gromov dans \cite{Gromov-FRM}, qui est la suivante. On part d'un cycle géométrique $(V,f,g)$ représentant la classe d'homologie $h$ dans $H_n(\pi;\Z)$ tel que, pour $\eps_1>0$, on ait :
$$\sigma(V,f,g)\le\sigma(h)+\eps_1.$$
Le but est alors de couper les \og{}mauvaises boules\fg{} dans ce cycle. Mais il n'y a pas de procédure constructive pour faire cela. 
On va, dans un premier temps, plonger $V$ dans un espace $K(V)$ (appelé extension du cycle $V$) qui possède de belles propriétés :

\begin{quote}
	\begin{itemize}
		\item On y contrôle la systole relative de certains cycles  géo\-métriques qui représentent $h$ ;
		\item Les cycles dans $K(V)$ vérifient une bonne inégalité iso\-péri\-métrique (voir théorème \ref{theo-inégalitéisocomplexes}), qui est une conséquence de l'inégalité de remplissage de Gromov dans les espaces de Banach (voir \cite{Gromov-FRM}, \cite{Wenger-SP} ou \cite{Guth-Note}).\\
	\end{itemize}
\end{quote}

La première partie de ce texte est consacrée à la construction de l'espace $K(V)$. Rappelons que si $X$ est un espace métrique compact, l'application $x\mapsto \dist(.,x)$ est un plongement isométrique de $X$ dans l'espace de Banach $\mathcal B(X,\R)$ des fonctions bornées sur $X$, muni de la norme de la convergence uniforme (c'est le plongement de Kuratowsky). La construction de l'image $V'$ de $V$ dans $K(V)$ est une variante de ce plongement. Cependant l'extension $K(V)$, qui est compacte, permettra beaucoup plus de souplesse que de travailler dans $\mathcal B(V,\R)$. \\

La suite de la démonstration consiste à montrer le résultat intermédiaire suivant.

\begin{theoremalph}
\label{theoB}
Soient $h$ une classe d'homologie non nulle dans $H_n(\pi;\Z)$, $\eps>0$ et $(V,f,g)$ un cycle géométrique de dimension $n\ge 2$ qui représente $h$ tel que $\sigma(V,f)\le \sigma(h)+\eps$. Soit $a$ dans $]0,\tfrac 12\:\sys(V,f,g)[$. Il existe une suite de cycles géométriques $(V_i,f_i,g_i)$ dans $K(V)$ qui représentent $h$ tels que :
\begin{enumerate}
	\item $\vol(V_i)\underset{i\to +\infty}{\longrightarrow}\vol[V']$ (où $V'$ est l'image de $V$ dans $K(V)$);
	\item Pour $R\in [a,\tfrac 12\:\sys(V,f,g)]$, les boules $B(R)$ de rayon $R$ dans chaque $V_i$ vérifient :
	$$\vol (B(R))\ge A_n (R-a)^n$$
	pour une certaine constante universelle $A_n$ qui ne dépend que de la dimension de $h$.
\end{enumerate}
\end{theoremalph}

Pour démontrer ce résultat, on considère une suite minimisante, pour le volume homologique de $V'$, de pseudo-variétés dans $K(V)$. De cette suite, par un argument de compacité, on montre que l'on peut en extraire la suite de cycles géométriques  $(V_i,f_i,g_i)$ du théorème \ref{theoB}.

Pour terminer la démonstration du théorème A, il ne reste plus qu'à remplacer, dans le second point du théorème \ref{theoB}, $\sys(V,f,g)$ par $\sys(V_i,f_i,g_i)$.

%%%%%%%%%%%%%%%%%%%%%%%%%%%%%%%%%%%%%%%%%%%%%%%%%%%%%%%%%%%%%%%%%%%%%%%%%%%%%%%%%%%
%%%%%%%%%%%%%%%%%%%%%%%%%%%%%%%%%%%%%%%%%%%%%%%%%%%%%%%%%%%%%%%%%%%%%%%%%%%%%%%%%%%
%%%%%%%%%%%%%%%%%%%%%%%%%%%%%%%%%%%%%%%%%%%%%%%%%%%%%%%%%%%%%%%%%%%%%%%%%%%%%%%%%%%
\subsection{Pseudo-variétés et volumes des chaînes}
\label{vocabulaire}
Avant de commencer, rappelons rapidement la signification de quelques notions utilisées dans ce texte.
Pour plus de précisions, on peut consulter \cite{Spanier95}.\\ 

Une \textit{pseudo-variété}\index{Pseudo-variété} de dimension $n$ est un complexe simplicial fini $K$ tel que :

\begin{itemize}
	\item $\dim K=n$;
	\item Chaque simplexe dans $K$ est face d'un simplexe de dimension $n$ (homogénéité de la dimension);
	\item Chaque simplexe de dimension $n-1$ est face d'exactement deux simplexes de dimension $n$ (pas de bifurcation);
	\item Dès que $\sigma$ et $\tau$ sont deux simplexes distincts de dimension $n$ dans $K$, il existe une suite $\sigma_1=\sigma,\ldots,\sigma_p=\tau$ de simplexes de dimension $n$ dans $K$ tels que pour tout $i$ dans $\set{1,p-1}$ les simplexes $\sigma_i$ et $\sigma_{i+1}$ aient une face de dimension $n-1$ commune (forte connexité).\\
\end{itemize}

Il est équivalent de dire qu'un complexe simplicial $K$ est une pseudo-variété de dimension $n$ lorsque $K$ est de dimension $n$ et lorsqu'il existe un sous-complexe $\Sigma\subset K$ vérifiant les trois propriétés suivantes :

\begin{itemize}
	\item $\dim \Sigma \le n-2$;
	\item $K\setminus \Sigma$ est une variété topologique de dimension $n$ dense dans $K$;
	\item L'espace $K\setminus \Sigma$ est connexe.\\
\end{itemize}
Dans la suite, toutes les pseudo-variétés considérées sont supposées orientables.\\

Définissons maintenant la notion de \textit{métrique riemannienne sur un polyèdre}\index{Métrique polyédrale}. On suit pour cela \cite{Babenko-TSU} ou \cite{Babenko-FSI}. Soit $P$ un polyèdre (espace topologique muni d'une triangulation) fini.  Une métrique riemannienne sur $P$ est  une famille de métrique $(g_\sigma)_{\sigma\in \Sigma}$, où $\Sigma$ est l'ensemble des simplexes de $P$, qui vérifie :
\begin{quote}
\begin{itemize}
	\item Chaque $g_\sigma$ est une métrique riemannienne lisse à l'intérieur de $\sigma$ ;
	\item Dès que $\sigma_1$ et $\sigma_2$ sont dans $\Sigma$, on a l'égalité $$g_{\sigma_1}\mid_{ \sigma_1\cap \sigma_2}=g_{\sigma_2}\mid_{\sigma_1\cap \sigma_2}.$$
\end{itemize}
\end{quote}

Un polyèdre riemannien est alors un espace de longueur pour la distance induite par cette famille de métriques. Pour simplifier les notations, on désignera par une seule lettre $g$ la famille $(g_\sigma)_{\sigma\in \Sigma}$ et on dira que $g$ est une métrique riemannienne lisse par morceaux sur $P$. Le volume $\vol(P,g)$ est alors correctement défini de manière évidente. \\

On aura besoin aussi de définir une notion de volume pour les cycles (lipschitziens) dans un espace métrique $(X,\dist)$ (voir \cite{Gromov-FRM}, page 11). On note  $\Delta^q\subset \R^{q+1}$  le $q$-simplexe standard de $\R^{q+1}$.
 
\begin{defi}
\label{def-volumecycle}Soit $\sigma: \Delta^{q}\to X$ un simplexe (lipschitzien) dans $X$. Le volume de $\sigma$ dans $(X,\dist)$ est l'infimum des volumes de $\Delta^q$ muni d'une métrique riemannienne $\mathcal G$ telle que $\sigma$ soit une contraction de $(\Delta^q,\dist_{\mathcal G})$ sur $(X,\dist)$, où $\dist_{\mathcal G}$ est la distance induite par $\mathcal G$.

Lorsque $c=\dsum_{i\in I}k_i\sigma_i$ est une $q$-chaîne singulière\index{Volume des chaînes} ($I$ est une partie finie de $\N$, les $k_i$ sont des éléments de $\Z$), le volume de $c$ dans $(X,\dist)$ est alors :
$$\vol(c)=\dsum_{i\in I} \mod{k_i}\: \vol(\sigma_i).$$
\end{defi}

\begin{rema}

\begin{enumerate}
	\item Lorsque $A\subset X$, si $c$ est une chaîne dans $A$ alors son volume dans $(X,\dist)$ et dans $(A,\dist_A)$ coïncident, où $\dist_A$ est la restriction à $A^2$ de $\dist$.
	\item Supposons que $X$ soit un polyèdre, $\dist$ étant la distance induite par une métrique riemannienne $g$ sur $X$. Pour un cycle lipschitzien $\sigma:\Delta^q\to X$, le volume de $\sigma$ dans $(X,\dist)$ est alors $\vol(\sigma)=\vol(\Delta,\sigma^*g)=\vol(\sigma,g)$, où $\sigma^*g$ est la métrique \og{}tirée en arrière\fg{} sur $\Delta^q$ (rappelons que $\sigma$ est presque partout différentiable).
\end{enumerate}

\end{rema}
\begin{Rema-ess}

\label{rema-metrique}
La définition \ref{def-volumecycle} permet de définir le volume $\vol(c)$ d'une chaîne $c$ dans $(\R^N,\norm{\ }_\infty)$. Pour des sous-variétés de $\R^N$ cette notion de volume coïncide avec le volume hyper-euclidien, appelé aussi volume riemannien inscrit (voir \cite{Gromov-FRM} page 32). Précisons cela. 

Soit $V$ une sous-variété de dimension $q$ de $\R^N$. Le volume hyper-euclidien de $V$ est défini par :
	$$\vol_{he}(V)=\underset{g}{\inf} \vol(V,g),$$
l'infimum étant pris sur toutes les métriques riemanniennes $g$ sur $V$ telles que, sur chaque espace tangent, on ait $\norm{\ }_{g}\ge \norm{\ }_\infty$ (où $\norm{\ }_g$ est la norme euclidienne induite par $g$). Il existe d'ailleurs une unique métrique $g_{he}$ sur $V$pour laquelle cet infimum est atteint (voir par exemple \cite{Ivanov-volumes}). Ainsi $\vol_{he}(V)=\vol(V,g_{he})$.
 
Si maintenant $\sigma:\Delta^q\to V$ est un simplexe différentiable dont le jacobien est de rang maximum, alors son volume dans $(\R^N,\norm{\ }_\infty)$ n'est autre que $\vol(\sigma,g_{he})$.  
Puis la classe fondamentale de $V$ peut être représentée par un cycle $c_V=\dsum_{i=1}^p\sigma_i$ où $\sigma_1,\ldots,\sigma_p$ sont les simplexes de dimension $q$ de $V$, et il vient  :
$$\vol(c_V)=\dsum_{i=1}^p\vol(\sigma_i)=\dsum_{i=1}^p\vol(\sigma_i,g_{he})=\vol(V,g_{he}).$$

Cela est encore vrai lorsque $V$ est une pseudo-variété.\\ 

Dans la suite, on ne considérera que des cycles $c$ dans $\R^N$, pour un certain $N$ fixé, et $\vol(c)$ désignera le volume de $c$ dans $(\R^N,\norm{\ }_\infty)$. 
%\end{enonce}
\end{Rema-ess}
%%%%%%%%%%%%%%%%%%%%%%%%%%%%%%%%%%%%%%%%%%%%%%%%%%%%%%%%%%%%%%%%%%%%
%%%%%%%%%%%%%%%%%%%%%%%%%%%%%%%%%%%%%%%%%%%%%%%%%%%%%%%%%%%%%%%%%%%%
\section{Extension des cycles géométriques}
%%%%%%%%%%%%%%%%%%%%%%%%%%%%%%%%%%%%%%%%%%%%%%%%%%%%%%%%%%%%%%%%%%%%%%%%%%%%%%%%%%%
%%%%%%%%%%%%%%%%%%%%%%%%%%%%%%%%%%%%%%%%%%%%%%%%%%%%%%%%%%%%%%%%%%%%%%%%%%%%%%%%%%%
%%%%%%%%%%%%%%%%%%%%%%%%%%%%%%%%%%%%%%%%%%%%%%%%%%%%%%%%%%%%%%%%%%%%%%%%%%%%%%%%%%%
\subsection{Complexes cubiques}
%%%%%%%%%%%%%%%%%%%%%%%%%%%%%%%%%%%%%%%%%%%%%%%%%%%%%%%%%%%%%%%%%%%%%%%%%%%%%%%%%%%
%%%%%%%%%%%%%%%%%%%%%%%%%%%%%%%%%%%%%%%%%%%%%%%%%%%%%%%%%%%%%%%%%%%%%%%%%%%%%%%%%%%
On va d'abord définir la notion de \textit{complexe cubique}, que l'on utilisera pour la construction de l'extension $K(V)$ du cycle géométrique $(V,f,g)$.
On considère un ensemble non vide $X$, ainsi que l'espace de Banach $\left(\mathcal B(X,\R),\norm{\ }_\infty\right)$ des fonctions bornées sur $X$.

\begin{defi} 
\label{def-face}On appellera cube standard sur $X$ l'ensemble :
$$Cube(X)=\left\{\varphi\in \mathcal B(X,\R)\mid \left(\forall x\in X\right)\left(0\le \varphi(x)\le 1\right) \right\}$$
Pour $x$ dans $X$ les \textit{faces associées à $x$} sont les parties suivantes de $Cube(X)$:
$$F_x^0=\left\{\varphi\in Cube(X)\mid \varphi(x)=0\right\}\ \ \mbox{et}\ \ F_x^1=\left\{\varphi\in Cube(X)\mid \varphi(x)=1\right\}.$$
\end{defi}

Remarquons que si $X$ est de cardinal 1 alors $Cube(X)$ est isométrique à $[0,1]$ et si $X$ est de cardinal $N$, $Cube(X)$ est isométrique à $[0,1]^N$. On peut définir une notion plus générale de \textit{face} pour $cube(X)$ de la manière suivante. Par exemple, pour $x$ dans $X$, il existe une isométrie $\psi_x:Cube(Y)\to F_x^0$ où $Y=X\setminus \left\{x\right\}$. Les images par $\psi_x$ des faces de $Cube(Y)$ seront encore appelées faces de $Cube(X)$. On peut d'ailleurs itérer cette définition tant que $Y$ est de cardinal supérieur à 1.

\begin{defi} Soit $E$ un espace métrique.  Lorsque $X$ est un ensemble, un \textit{cube sur $X$} dans  $E$ est une partie $K$ de $E$ telle qu'il existe une isométrie $\psi:Cube(X)\to K$; les \textit{faces} de $K$ sont alors les images par $\psi$ des faces de $cube(X)$.
On dira encore que $E$ est un \textit{complexe cubique}\index{Complexe cubique} lorsque l'on peut écrire $$E=\bigcup_{i\in I} K_i$$
où les $K_i$ (appelés \textit{briques} de $E$) sont des cubes sur un certain ensemble $X_i$ qui vérifient la contrainte suivante : pour  $i\neq j$ l'intersection  $K_i\cap K_j$ est soit une face (au sens généralisé) commune à $K_i$ et $K_j$, soit le vide. \\
Un \textit{sous-complexe cubique} d'un complexe cubique $K$ est un complexe cubique inclus dans $K$ dont les briques sont des faces de $K$.
\end{defi}

%%%%%%%%%%%%%%%%%%%%%%%%%%%%%%%%%%%%%%%%%%%%%%%%%%%%%%%%
\begin{exem}
	\begin{enumerate}
		\item  Soient $m$ un entier relatif, $X$ un ensemble et $x\in X$. L'ensemble $K$
	des $\varphi \in \mathcal B(X,\R)$ telles que $0\le \varphi(x)\le 1$ et $m\le \varphi(y)\le m+1$ pour tout $y$ dans $X\setminus \left\{x\right\}$ est un cube sur $X$. En effet, c'est l'image isométrique de $Cube(X)$ via l'application $\Phi:\varphi\mapsto \psi $
	où $\psi(x)=\varphi(x)$ et $\psi(y)=\varphi(y)+m$ pour tout $y$ dans $X\setminus \left\{x\right\}$.
	\item $[0,1]\times \left\{0\right\}$ est un cube sur $Y=0$ dans $\R^2$ et c'est une face du cube (sur $X=\left\{\ast,\Box\right\}$) $[0,1]^2$.
	\item Soit $X$ un ensemble non vide. Pour $x$ dans $X$, l'ensemble $F$
	des $\varphi \in \mathcal B(X,\R)$ telles que $\varphi(x)=0$ et $m\le \varphi(y)\le m+1$ pour tout $y$ dans $X\setminus \left\{x\right\}$ est une face de $Cube(X)$.
	\item $\R^N$ est un complexe cubique sur tout ensemble $X$ de cardinal $N$\ldots
		\item Pour tout ensemble $X$, $\mathcal B(X,\R)$ est un complexe cubique sur $X$ : la famille d'hyperplans affines $\left(\left\{\varphi\in \mathcal B(X,\R)\mid \varphi(x)=m\right\}\right)_{(x,m)\in X\times \Z}$ découpe $\mathcal B(X,\R)$ en cubes sur $X$.
	\item Lorsque $X$ est un espace topologique, toutes les constructions précédentes sont encore valables si on remplace le Banach $(\mathcal B(X,\R),\norm{\ }_\infty)$ par le Banach $(L^\infty(X),\norm{\ }_\infty)$ des fonctions boréliennes bornées sur $X$.
	\end{enumerate}
\end{exem}
%%%%%%%%%%%%%%%%%%%%%%%%%%%%%%%%%%%%%%%%%%%%%%%%%%%%%%%%
%%%%%%%%%%%%%%%%%%%%%%%%%%%%%%%%%%%%%%%%%%%%%%%%%%%%%%%%

\begin{rema}
\label{rema-topologiecomplexe}
Dans la suite, quelque soit le complexe cubique inclus dans un espace $L^\infty$ considéré, on supposera qu'il est muni de la métrique de longueur associée à la distance induite par $\norm{\ }_\infty$. En particulier, pour $\varphi$ et $\psi$ dans un tel complexe cubique, on a aura $\norm{\varphi-\psi}_\infty\le \dist(\varphi,\psi)$. 
\end{rema}

On aura besoin par la suite du résultat suivant.

\begin{theo}
\label{theo-retraction}
Soit $\eps$ dans  $]0,\tfrac 12[$. Pour tout complexe cubique $K$, il existe une application affine par morceaux $R_\eps:K\to K$ qui envoie le $\eps$-voisinage de tout sous-complexe cubique $K'$ de $K$ sur $K'$.
\end{theo}
\begin{proof}[Démonstration]
Regardons ce qui se passe pour le complexe cubique $[0,1]$. On prend l'application $r_\eps:[0,1]\to [0,1]$ dont le graphe est donné par la figure \ref{fig-rétraction}.

%%%%%%%%%%%%%%%%%%%%%%%%%%%%%%%%%%%%%%%%%%%%%%%%%%%
%%%%%%%%%%%%%%%%%%%%%%%%%%%%%%%%%%%%%%%%%%%%%%%%%%%
% FIGURE RETRACTION
%%%%%%%%%%%%%%%%%%%%%%%%%%%%%%%%%%%%%%%%%%%%%%%%%%%
%%%%%%%%%%%%%%%%%%%%%%%%%%%%%%%%%%%%%%%%%%%%%%%%%%%

\begin{figure}[h]
\begin{center}
\includegraphics[width=0.4\linewidth]{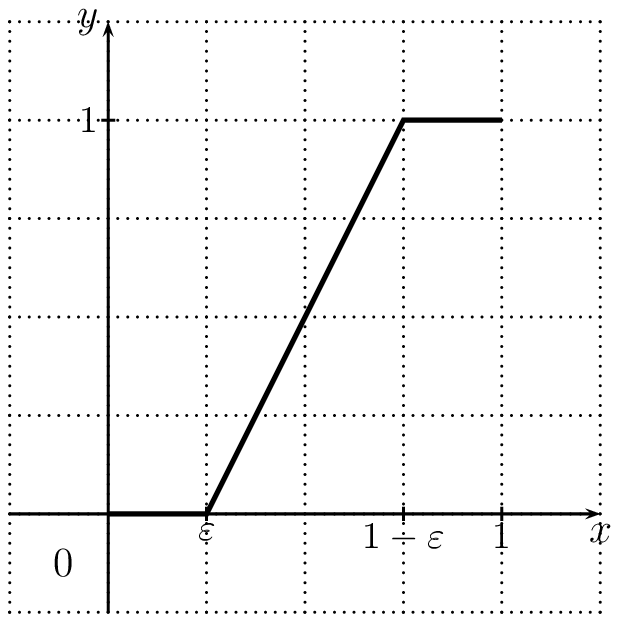}
\caption{L'application $r_\eps$} 
\label{fig-rétraction}
\end{center}
\end{figure}

%%%%%%%%%%%%%%%%%%%%%%%%%%%%%%%%%%%%%%%%%%
Si maintenant $K$ est le cube $[0,1]^N$, on prend $R_\eps(x)=(r_\eps(x_1),\ldots,r_\eps(x_N))$. De manière plus générale, si $X$ est un ensemble non vide et si $K=Cube(X)$, on prend $$R_\eps:\varphi\mapsto r_\eps\circ \varphi.$$
Dans le cas le plus général où $K=\displaystyle \bigcup_{i\in I} K_i$, on note $\psi_i:K_i\to cube(X_i)$ une isométrie. Dès que $\varphi\in K_i$, on pose alors : $R_\eps(\varphi)=r_\eps\circ \psi_i(\varphi)$.
\end{proof}

%%%%%%%%%%%%%%%%%%%%%%%%%%%%%%%%%%%%%%%%%%%%%%%%%%%%%%%%%%%%%%%%%%%%%%%%%%%%%%%%%%%
\subsection{Des complexes cubiques particuliers}
%%%%%%%%%%%%%%%%%%%%%%%%%%%%%%%%%%%%%%%%%%%%%%%%%%%%%%%%%%%%%%%%%%%%%%%%%%%%%%%%%%%
\label{subsection-complexecubique}
On  considère un cycle géo\-métrique $(V,f,g)$ ainsi qu'une partie finie $V_0=\left\{v_1,\ldots,v_N\right\}$ de $V$. Pour tout $i$ dans $\left\{1,\ldots,N\right\}$ on considère l'application $x_i: V\to [0,1]$ définie par :
$$x_i(v)=\min(\dist(v,v_i),1)$$
ainsi que l'application $I_0:V \to \R^N$ définie par :
$$I_0(v)=(x_1(v),\ldots,x_N(v))$$
%%%%%%%%%%%%%%%%
\begin{rema}  Pour $\eta>0$, on peut trouver $V_0$ suffisamment dense dans $V$ tel que, pour tout $(v,v')\in V^2$ avec $\dist(v,v')<\tfrac 1 2$, on ait :
$$\dist(v,v')-\eta\le \norm{I_0(v)-I_0(v')}_\infty\le \dist(v,v').$$
Lorsque $V$ est une variété riemannienne, on peut conclure (voir \cite{Guth-Note}) que : 
$$(1-\eta)\dist(v,v')\le \norm{I_0(v)-I_0(v')}_\infty\le \dist(v,v').$$
\end{rema}
%%%%%%%%%%%%%%%%

Construisons maintenant un complexe cubique à partir de $V$ et d'une partie finie $V_0$ de $V$ de cardinal $N$. Pour tout $v$ dans $V$ notons  $K(v)$ la face de plus petite dimension de $[0,1]^N$ qui contient $R_\eps\circ I_0(v)$. 
Par exemple si $N=3$, en notant $J=R_\eps\circ I_0$ on a  :
\begin{quote}

\begin{itemize}
	\item  Si $J(v)=(\ast,0,0)$ avec $\ast\in [\eps,1-\eps]$ alors $K(v)=[0,1]\times \left\{0\right\}\times \left\{0\right\}$. 
\item Si $J(v)=(\ast,\ast',0)$ avec $\ast\in [\eps,1-\eps]$ et $\ast'\in [\eps,1-\eps]$ alors $K(v)=[0,1]\times [0,1] \times \left\{0\right\}$. 
\end{itemize}

\end{quote}

On prend maintenant l'union de toutes ces faces $K(v)$, pour $v$ décrivant $V$. 

\begin{defi}
Le  sous-complexe cubique de $[0,1]^N$ constitué de l'union de tous les cubes $K(v)$ pour $v$ décrivant $V$ s'appelle \textit{une extension du cycle géométrique}\index{Extension d'un cycle géométrique} $(V,f,g)$. Ce complexe cubique sera noté $K(V)$ :
$$K(V)=\bigcup_{v\in V}K(v).$$
\end{defi} 

\begin{rema}
Ce complexe cubique $K(V)$ dépend de $V$, mais aussi de $V_0$ et $\eps<\tfrac1 2$. Il sera intéressant lorsque $V_0$ est $\alpha$-dense dans $V$ pour $\alpha$ suffisamment petit, où $\alpha$-dense\index{Partie $\alpha$-dense} signifie que tout point de $V$ se trouve à une distance inférieure à $\alpha$ d'un point de $V_0$. \\
\end{rema}

Lorsque $V_0=\left\{v_1,\ldots,v_N\right\}$ est suffisamment dense dans $V$, $K(v)$ est un complexe (cellulaire) de dimension au plus $N-1$. 

Plus précisément, pour $i$ dans $\left\{1,\ldots,N\right\}$, notons $K_i$ le cube défini par $x_i=0$ i.e. :
$$K_i=\left\{(x_1,\ldots,x_N)\in [0,1]^N\mid x_i=0\right\}.$$

Chaque $K_i$ contient $J(v_i)$. Si maintenant $V_0$ est $\alpha$-dense dans $V$ avec $\alpha\le \eps$, pour tout $v$ dans $V$ il existe alors $j$ dans $\left\{1,\ldots,N\right\}$ tel que $\dist(v,v_j)\le \eps$. Il en résulte que $J(v)\in K_j$ et ainsi : $$K(V)\subset \bigcup_{i=1}^N K_i.$$

L'espace $K(V)$ est donc inclus dans une union de faces de dimension $N-1$ de $[0,1]^N$. Mieux, sous les conditions précédentes, dès que $m$ est un point de $K(V)$, $m$ admet au moins une coordonnée nulle.

%%%%%%%%%%%%%%%%%%%%%%%%%%%%%%%%%%%%%%%%%%%%%%%%%%%%%%%%
\begin{rema}
\label{rema-complexecubique}
\begin{enumerate}
	\item Avec les notations ci-dessus, si $v$ et $v'$ dans $V$ sont tels que $J(v)$ et $J(v')$ vivent dans un même cube $K_i$, pour un certain $i$ dans $\left\{1,\ldots,N\right\}$, alors on a $\dist(v,v')\le 2\eps$ et, en particulier, $\dist(v,v')< 1$.
	\begin{quote}
	En effet, puisque $J(v)$ et $J(v')$ sont dans $K_i$, on a  $$r_\eps\big(\min(1,\dist(v,v_i))\big)=0=r_\eps\big(\min(1,\dist(v',v_i))\big),$$
ce qui impose $\dist(v,v_i)\le \eps$ et $\dist(v',v_i)\le \eps$; il vient alors $\dist(v,v')\le 2\eps <1$. \\
	\end{quote}

\item Lorsque $V_0$ est  une partie infinie de $V$, la même construction s'adapte. Pour $v$ dans $V$, $K(v)$ est alors la plus petite face de $Cube(V_0)$ qui contient $J(v)$. Dans ce cas $K(V)$ est une union de faces de $Cube(V_0)\subset\mathcal B(V_0,\R)$, et, lorsque $V_0$ est suffisamment dense dans $V$, chaque $K(v)$ est contenu dans une face $F_w^0$ pour un certain $w$ dans $V_0$ (pour les notations, voir la définition \ref{def-face} page \pageref{def-face}).
\end{enumerate}
\end{rema}

\begin{rema}
Dans \cite{Gromov-FRM} § 6.2, Gromov a introduit la notion de $\delta$-complexe cubique pour $\delta>0$ : dans la définition de $I_0$, on remplace $x_i(v)$ par $\min(\dist(v,v_i),\delta)$. Le complexe cubique obtenu $K_\delta(V)\subset [0,\delta]^N$  dépend alors en plus du choix de $\delta$.
\end{rema}
%%%%%%%%%%%%%%%%%%%%%%%%%%%%%%%%%%%%%%%%%%%%%%%%%%%%%%%%%%%%%%%%%%%%%%%%%%%%%%%%%%%
\subsection{Plongement d'un cycle géométrique dans son extension}
%%%%%%%%%%%%%%%%%%%%%%%%%%%%%%%%%%%%%%%%%%%%%%%%%%%%%%%%%%%%%%%%%%%%%%%%%%%%%%%%%%%
Considérons maintenant l'application $J:V\to K(V)$
définie par 
\begin{equation}
%\label{defplongement} 
	J=R_\eps\circ I_0
\end{equation}
On peut noter que $J$ est lipschitzienne de rapport $\dfrac{1}{1-2\eps}$.

%%%%%%%%%%%%%%%%%%%%%%%%%%%%%%%%%%%%%%%%%%%%%%%%%%%%%%%%%%
\begin{exem} Prenons pour $V$ un cercle de périmètre $1$ et $V_0=\left\{\tfrac 14,\tfrac 12,\tfrac 34\right\}$. Ainsi $V_0$ est $\alpha$-dense dans $V$ pour $\alpha=\tfrac 14$. 
La figure \ref{fig-cercle1} montre alors $J(V)$ en gras ainsi que $K(V)$ qui est hachuré, pour $\eps=\tfrac 14$.

%%%%%%%%%%%%%%%%%%%%%%%%%%%%%%%%%%%%%%%%%%%%%%%%%%%
%%%%%%%%%%%%%%%%%%%%%%%%%%%%%%%%%%%%%%%%%%%%%%%%%%%
% FIGURE PLONG CERCLE 1
%%%%%%%%%%%%%%%%%%%%%%%%%%%%%%%%%%%%%%%%%%%%%%%%%%%
%%%%%%%%%%%%%%%%%%%%%%%%%%%%%%%%%%%%%%%%%%%%%%%%%%%

\begin{figure}[h]
\begin{center}
\includegraphics[width=0.6\linewidth]{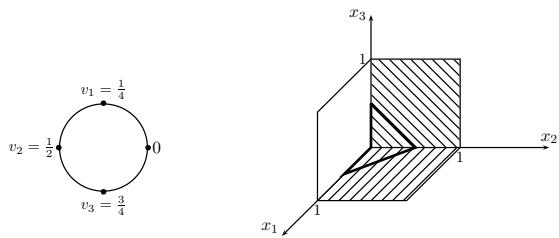}
\caption{Plongement du cercle $V$ de périmètre 1 dans $K(V)$.}
\label{fig-cercle1}
\end{center}
\end{figure}

En prenant pour $V$ un cercle de périmètre 2 et $V_0=\left\{0,\tfrac 23,\tfrac 43\right\}$ on obtient alors la figure \ref{fig-cercle2}, où l'image de $V$ et $K(V)$ sont confondues (avec $\eps=\tfrac 13$).

%%%%%%%%%%%%%%%%%%%%%%%%%%%%%%%%%%%%%%%%%%%%%%%%%%%
%%%%%%%%%%%%%%%%%%%%%%%%%%%%%%%%%%%%%%%%%%%%%%%%%%%
% FIGURE PLONG CERCLE 2
%%%%%%%%%%%%%%%%%%%%%%%%%%%%%%%%%%%%%%%%%%%%%%%%%%%
%%%%%%%%%%%%%%%%%%%%%%%%%%%%%%%%%%%%%%%%%%%%%%%%%%%

\begin{figure}[h]
\begin{center}
\includegraphics[width=0.5\linewidth]{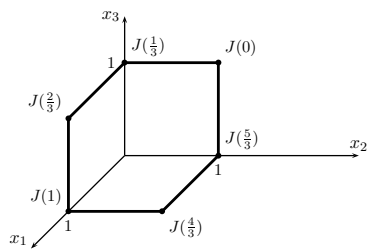}
\caption{Plongement du cercle $V$ de périmètre 2 dans $K(V)$.} 
\label{fig-cercle2}
\end{center}
\end{figure}

\end{exem}

Ces exemples élémentaires suggèrent que $J:V\to K(V)$ est un plongement, lorsque $V_0$ est suffisamment dense dans $V$.

%%%%%%%%%%%%%%%%%%%%%%%%%%%%%%%%%%%%%%%%%%%%%%%%%%%%%%%%%%%%%%%%%%%%%%%%%%%%%%%%%%%
\begin{theo}
Pour $\eps$ suffisamment petit et $V_0$ suffisamment dense dans $V$, l'application $J$ est injective.
\end{theo}
%%%%%%%%%%%%%%%%%%%%%%%%%%%%%%%%%%%%%%%%%%%%%%%%%%%%%%%%%%%%%%%%%%%%%%%%%%%%%%%%%%%
%%%%%%%%%%%%%%%%%%%%%%%%%%%%%%%%%%%%%%%%%%%%%%%%%%%%%%%%%%%%%%%%%%%%%%%%%%%%%%%%%%%
%%%%%%%%%%%%%%%%%%%%%%%%%%%%%%%%%%%%%%%%%%%%%%%%%%%%%%%%%%%%%%%%%%%%%%%%%%%%%%%%%%%

\begin{proof}[Démonstration]
Soient $v$ et $v'$ dans $V$. Remarquons que, si $J(v)=J(v')$, alors, pour tout $i$ dans $\set {1,\ldots,N}$ tel que $\dist(v,v_i)\in [\eps,1-\eps]$, on a $\dist(v_i,v)=\dist(v_i,v')$. Supposons que $\dist(v,v')\ge 2\eps$. Il existe  $v_i$  dans $V_0$ tel que $\dist(v_i,v)\le \alpha$. Si $0<\alpha<\eps<\tfrac12$ on a immédiatement $\dist(v',v_i)>\eps$ et ainsi $r_\eps(x_i(v))=0\neq r_\eps(x_i(v'))$. Il en résulte que $J(v)\neq J(v')$.
  
Supposons maintenant que $V$ soit un simplexe euclidien de dimension $n$ et qu'il existe deux points distincts $v$ et $v'$ dans $V$ qui vérifient $J(v)=J(v')$. Si $w$ est un point de  $V_0$ tel que $\dist(v,w)\in [\eps,1-\eps]$ on a  $\dist(v,w)=\dist(v',w)$, donc tous les points de $V_0$ situés à une distance comprise entre $\eps$ et $1-\eps$ de $v$ appartiennent alors à l'hyperplan médiateur de $v$ et $v'$. C'est une contradiction lorsque $\eps$ est suffisamment petit et $V_0$ suffisamment dense dans $V$.

%%%%%%%%%%%%%%%%%%%%%%%%%%%%%%%%%%%%%%%%%%%%%%%%%%%
%%%%%%%%%%%%%%%%%%%%%%%%%%%%%%%%%%%%%%%%%%%%%%%%%%%
% FIGURE SIMPLEXE
%%%%%%%%%%%%%%%%%%%%%%%%%%%%%%%%%%%%%%%%%%%%%%%%%%%
%%%%%%%%%%%%%%%%%%%%%%%%%%%%%%%%%%%%%%%%%%%%%%%%%%%

\begin{figure}[h]
\begin{center}
\includegraphics[width=0.5\linewidth]{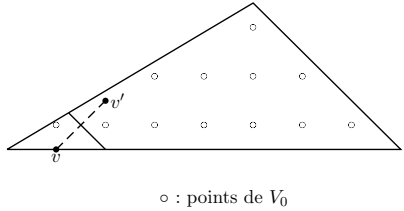}
\caption{Plongement d'un simplexe.} 
\end{center}
\end{figure}

%%%%%%%%%%%%%%%%%%%%%%%%%%%%%%%%%%%%%%%%%%%%%%%%%%%

Supposons maintenant que $V$ soit un polyèdre euclidien. Soient  $v$ et $v'$ deux points distincts dans $V$ qui vérifient $\dist(v,v')\le 2\eps$. Appelons $\sigma$ un simplexe de dimension $n$ contenant $v$. L'ensemble des points $w$ de $\sigma$ tels que $\dist(v,w)=\dist(v',w)$ est de dimension inférieure à $n-1$. Ainsi, si on peut trouver $w$ dans $\sigma$ tel que $\dist(w,v)=3\eps$ et $B(w,2\eps)\cap \sigma$ soit non vide, pour $V_0$ suffisamment dense dans $V$, il y aura dans cette boule un point de $V_0$ qui distinguera $v$ et $v'$. Mais pour $\eps$ suffisamment petit, on peut toujours trouver une telle boule.\\

Enfin, par exemple d'après \cite{Cheeger2}, lorsque $V$ est quelconque on peut l'approcher par des polyèdres euclidiens tout en gardant systole relative et volume aussi proche que l'on veut.
\end{proof}

\begin{Rema-ess}
\label{rem-essentielle}
Désignons par $\psi:V'\to V$ l'homéomorphisme réciproque de $J:V\to V'=J(V)$. On a alors, au niveau des groupes d'homologie de dimension $n$ :
$$f_*\circ \psi_*[V']=h.$$
Ainsi, si $g_0$ est la métrique sur $V'$ fournie par la remarque \ref{rema-metrique}, $(V',f\circ\psi,g_0)$ est un cycle géométrique qui représente $h$. 
$$\xymatrix{  V' \ar[rr]^{\psi} \ar[rd]_{f'}&&  V \ar[ld] ^{f}\\ &  K(\pi,1) }$$

Notons, pour la suite, $f'=f\circ\psi$. 
\end{Rema-ess}

%%%%%%%%%%%%%%%%%%%%%%%%%%%%%%%%%%%%%%%%%%%%%%%%%%%%%%%%%%%%%%%%%%%%%%%%%%%%%%%%%%%
\subsection{Une inégalité isopérimétrique dans $K(V)$}
%%%%%%%%%%%%%%%%%%%%%%%%%%%%%%%%%%%%%%%%%%%%%%%%%%%%%%%%%%%%%%%%%%%%%%%%%%%%%%%%%%%
Les cycles des complexes cubiques satisfont à une inégalité isopérimétrique remarquable, qui jouera un rôle essentiel dans la démonstration du résultat principal. Avant d'énoncer ce résultat, il est nécessaire de préciser certaines notions.

Lorsque $z$ est un cycle de dimension $n$ dans un espace métrique, le \textit{volume de remplissage} de $z$ dans $X$ (comparer à \cite{Gromov-FRM} page 12) est :
$$\volremp(z\subset X)=\inf \left\{\vol(c)\mid \mbox{$c$ chaîne de dimension $n+1$ dans $X$
 de bord $z$}
\right\}$$
%%%%%%%%%%%%%%%%%%%%%%%%%%%%%%%%%%%%%%%%%%%%%%%%%%%%%%%%%%%%%%%%%%%%%%%%%%%%%%%%%%%
\begin{theo}
\label{theo-inégalitéisocomplexes}
Soit $K(V)$ une extension d'un cycle géométrique $(V,f,g)$. Il existe deux constantes $\alpha_n$ et $\beta_n$ telles que tout cycle $z$ dans $K(V)$ de dimension $n$ et de volume $\vol(z)\le \alpha_n$ soit le bord d'une chaîne $c$ dans $K(V)$ de dimension $n+1$ qui vérifie :
\begin{enumerate}
	\item $\vol(c)\le \beta_n\:(\vol\:z)^{\frac{n+1}{n}}$ donc $\volremp (z)\le \beta_n\:(\vol\:z)^{\frac{n+1}{n}}$.
	\item $c$ est contenue dans le voisinage tubulaire de rayon $\eps_n=\beta_n\:(\vol\:z)^{\frac 1n}$ de $z$.
\end{enumerate}
\end {theo}
%%%%%%%%%%%%%%%%%%%%%%%%%%%%%%%%%%%%%%%%%%%%%%%%%%%%%%%%%%%%%%%%%%%%%%%%%%%%%%%%%%%

%%%%%%%%%%%%%%%%%%%%%%%%%%%%%%%%%%%%%%%%%%%%%%%%%%%%%%%%%%%%%%%%%%%%%%%%%%%%%%%%%%%
\begin{proof}[Démonstration]
Soit $z$ un cycle singulier de dimension $n$ dans $K(V)\subset \R^N$. 
Un résultat de Gromov (voir \cite{Gromov-FRM} \textbf{4.2} et \textbf{4.3}), dont une preuve détaillée se trouve dans dans \cite{Guth-Note}, assure l'existence d'une constante $C_n$ (ne dépendant que de la dimension du cycle $z$) et d'une chaîne $c_1$  dans $\R^N$ telle que $\partial c_1=z$ qui vérifie :
\begin{quote}
\begin{enumerate}
	\item $\vol(c_1)\le C_n\:(\vol\:z)^{\frac{n+1}{n}}$;
	\item $c_1$ est contenue dans le voisinage tubulaire $\tub(z,\eps_n)$ de rayon $\eps_n=(n+1)C_n\:(\vol\:z)^{\frac 1{n}}$ de $z$ dans $\R^N$.
\end{enumerate}
\end{quote}

On peut aussi comparer ce résultat à l'énoncé donné dans \cite{Gromov-MS} page 266.
Soit $\alpha_n>0$ tel que la relation $\vol(z)\le \alpha_n$ implique que $\eps=\eps_n=(n+1)C_n\:(\vol\:z)^{\frac 1{n}}<\tfrac 14$. On peut prendre par exemple :

\begin{equation}
	\label{alpha}
	\alpha_n=\dfrac1{4^{n}\left(n+1\right)^n C_n^n}
\end{equation}
 On peut alors considérer la rétraction $R_{\eps}$ du théorème \ref{theo-retraction}, pour le complexe cubique $\R^N$, qui fournit une chaîne $c_2=R_{\eps}(c_1)$ de dimension $n+1$ dans $K(V)$. Mais on a $\partial R_\eps(c_1)=R_{\eps}(\partial c_1)=R_{\eps}(z)$, donc cette chaîne ne convient pas tout à fait.

Notons que, pour tout $t$ dans $\R$, on a $\mod{r_\eps(t)-t}\le\eps$. Ainsi, pour tout $x$ dans $\R^n$, $\norm{R_\eps(x)-x}\le \eps$. On considère, pour tout $x$ dans $z$, le segment $s_x$ de $\R^N$ d'extrémités $x$ et $R_\eps(x)$ : ce segment est de longueur inférieure à $\eps$. De plus, si $x$ est dans une face de $K(v)$ alors $R_\eps(x)$ est dans cette même face : il en va donc de même du segment $s_x$. On considère alors le cylindre $\mathcal C$ constitué par $z$, $R_\eps(z)$, ainsi que les segments $s_x$, pour $x$ décrivant $z$. Ce cylindre est une chaîne de dimension $n+1$ dans $K(V)$ dont le bord est $\partial \mathcal C=z-R_{\eps}(z)$. Son volume  vérifie l'inégalité grossière :
$$\vol(\mathcal C)\le \eps(\vol(z)+\vol(R_\eps(z))).$$
On pose maintenant $c=c_2+\mathcal C$. C'est une chaîne de dimension $n+1$ dans $K(V)$. On a :
$$\partial c=\partial c_2+\partial\mathcal C=z.$$
De plus, $R_\eps$ étant lipschitzienne de rapport $\frac 1{1-2\eps}\le 2$, il vient :
\eq
\vol(c)&\le& 2^{n+1}\vol(c_2)+\vol(\mathcal C)\\
&\le &C_n(\vol\:z)^{\frac {n+1}n}+(n+1)C_n(\vol\:z)^{\frac 1n}\left(\vol\:z+ 2^{n}\vol\:z\right)\\
&\le &C_n\left(2^{n+1}+(n+1)(1+2^{n})\right)(\vol\:z)^{\frac {n+1}n}.
\fineq
Ainsi $\vol(c)\le \beta_n\:(\vol\:z)^{\frac{n+1}{n}}$, où

\begin{equation}
\label{beta}
	\beta_n=C_n\left(2^{n+1}+(n+1)(1+2^{n})\right)
\end{equation}

Enfin, tout point de $\mathcal C$ est à distance inférieure à $\eps$ de $z$. Si maintenant $x$ est dans $c_1$ alors :
\eq
\dist(R_{\eps}(x),z)&\le& \dist(R_{\eps}(x),x)+\dist(x,z)\\
&\le &\eps +\eps=2\eps=2C_n(\vol\:z)^{\frac 1n}\\
&\le &\beta_n\vol(z)^{\frac 1n}
\fineq
Il en résulte que $c$ est contenu dans le voisinage tubulaire de rayon $\eps=\beta_n(\vol\:z)^{\frac 1n}$ de $z$ dans $K(V)$.
\end{proof}

\begin{rema}
Le même type de démonstration s'applique à un complexe cubique quelconque.
\end{rema}

%%%%%%%%%%%%%%%%%%%%%%%%%%%%%%%%%%%%%%%%%%%%%%%%%%%%%%%%%%%%%%%%%%%%
%%%%%%%%%%%%%%%%%%%%%%%%%%%%%%%%%%%%%%%%%%%%%%%%%%%%%%%%%%%%%%%%%%%%
\section{Lien avec les systoles relatives}
%%%%%%%%%%%%%%%%%%%%%%%%%%%%%%%%%%%%%%%%%%%%%%%%%%%%%%%%%%%%%%%%%%%%%%%%%%%%%%%%%%%
\subsection{Une autre définition de la systole relative}
%%%%%%%%%%%%%%%%%%%%%%%%%%%%%%%%%%%%%%%%%%%%%%%%%%%%%%%%%%%%%%%%%%%%%%%%%%%%%%%%%%%
%%%%%%%%%%%%%%%%%%%%%%%%%%%%%%%%%%%%%%%%%%%%%%%%%%%%%%%%%%%%%%%%%%%%%%%%%%%%%%%%%%%
%%%%%%%%%%%%%%%%%%%%%%%%%%%%%%%%%%%%%%%%%%%%%%%%%%%%%%%%%%%%%%%%%%%%%%%%%%%%%%%%%%%
Avant de poursuivre, on a besoin de quelques précisions sur les revêtements. 
On considère toujours un cycle géométrique $(V,f,g)$ ainsi que le revêtement galoisien $p:\w V\to V$ associé au sous-groupe distingué $\ker f_*$, où $f_*:\pi_1(V)\to \pi$ est l'application induite par $f$ au niveau des groupes fondamentaux (pour plus de détails, voir \cite{Hatcher}). Notons $q:\w {K(\pi,1)}\to \K$ le revêtement universel de $\K$; l'application $f:V\to \K$ est recouverte par une application continue $\w f:\w V\to \w {K(\pi,1)}$ telle que le diagramme suivant commute.

$$\begin{CD}
\w V @>\w f>>\w {K(\pi,1)} \\ @V p VV @VV q V\\ V@>>f> \K
\end{CD}$$

On notera $\Gamma=G(\w V)$ le groupe des automorphismes de ce revêtement, que l'on peut considérer comme un sous-groupe de $\pi$, puisqu'isomorphe à $\im f_*$.  Il existe  une unique distance $\w \dist $ sur $\w V$ pour laquelle $\Gamma$ agit par isométrie sur $\w V$ et telle que $f$ soit une isométrie locale (voir \cite{BBI} page 84). On a alors le résultat suivant.

%%%%%%%%%%%%%%%%%%%%%%%%%%%%%%%%%%%%%%%%%%%%%%%%%%%%%%%%%%%%%%%%%%%%%%%%%%%%%%%%%%%
\begin{theo}
\label{theo-defsystole}
Avec les notations ci-dessus :
\begin{equation}
\label{sysdep}
	\sys(V,f,g)=\inf \left\{\w\dist(\w v,\gamma\cdot\w v)\mid\mbox{$\w v\in\w V$ et $\gamma\in \Gamma \setminus\left\{1_\Gamma\right\}$}\right\}
\end{equation}
\end{theo}
%%%%%%%%%%%%%%%%%%%%%%%%%%%%%%%%%%%%%%%%%%%%%%%%%%%%%%%%%%%%%%%%%%%%%%%%%%%%%%%%%%%

\begin{proof}[Démonstration]
Prenons $\w v\in \w V$ et $\gamma$ dans ${\Gamma} \setminus \left\{1_{\Gamma} \right\}$. On pose $v=p(\w v)$. Prenons $\w c$ un chemin minimisant entre $\w v$ et $\gamma\cdot \w v$. On a donc :
$$\w \dist(\w v,\gamma\cdot\w v)=\Long(\w c)=\Long(c)$$
où $c$ est le lacet  $p\circ \w c$ de point base $v$. Montrons alors que $f\circ c$ n'est pas homotope à un point dans $\K$. Le lacet $f\circ c$ n'est autre que le lacet $q\circ \w f\circ \w c$. Mais $\w c(1)=\gamma\cdot \w v$ donc, $\w f$ étant équivariante :
$$\w f(\gamma\cdot \w v)=\gamma\cdot \w f(\w v).$$
Comme $\gamma\neq 1_\Gamma$, on a $\w f(\w v)\neq \gamma\cdot \w f(\w v)$ donc $f\circ c$ n'est pas homotope à un point.

Soit maintenant $c$ un lacet dans $V$ tel que $\Long(c)=\sys(V,f,g)$. On pose $v=c(0)$ et on prend un relevé $\w v$ de $v$ dans $\w V$. Il existe alors un unique chemin $\w c:[0,1]\to \w V$ tel que $\w c(0)=\w v$ et $p\circ \w c=c$. En notant $\gamma$ l'élément de  $\Gamma$ tel que $ \w v\cdot \gamma=\w c(1)$, on obtient :
$$\w \dist(\w v,\gamma\cdot\w v)\le \Long (\w c)=\sys(V,f,g),$$
ce qui achève la preuve de (\ref{sysdep}) puisque $\gamma\neq 1_ {\Gamma}$.
\end{proof}

\begin{rema}
\label{rema-finsler_versus_riemannien}
Ce théorème permet de définir la notion de systole relative dans un cadre plus large. Soient $V'$  une pseudo-variété et $f': V'\to K(\pi,1)$. Lorsque $V'$ est munie d'une métrique de longueur $d'$, la systole relative $\sys(V',f',d')$ a alors un sens. 
Supposons maintenant que $V'$  soit  incluse dans $K(V)$. Notons $d_\infty$ la distance de longueur induite sur $V'$ par la norme $\norm{\ }_\infty$. Notons encore $d'$ la métrique de longueur induite par la métrique riemannienne lisse par morceaux sur $V'$ fournie par la remarque \ref{rema-metrique}. On a alors $d_\infty\le d'$ et ainsi $\sys(V',f',d_\infty)\le \sys(V',f',d')$. 
\end{rema}

%%%%%%%%%%%%%%%%%%%%%%%%%%%%%%%%%%%%%%%%%%%%%%%%%%%%%%%%%%%%%%%%%%%%%%%%%%%%%%%%%%%
\subsection{Systole relative des cycles géométriques dans \boldmath$K(V)$\unboldmath}
%%%%%%%%%%%%%%%%%%%%%%%%%%%%%%%%%%%%%%%%%%%%%%%%%%%%%%%%%%%%%%%%%%%%%%%%%%%%%%%%%%%
On suppose construite une extension $K(V)$ du cycle géométrique $(V,f,g)$ pour laquelle $J:V\to K(V)$ est un plongement, à partir d'une partie $V_0$ suffisamment dense de $V$, et on note comme précédemment $V'=J(V)$.
On va montrer que l'on peut contrôler la systole relative de certains cycles géométriques inclus dans $K(V)$ et qui représentent $h$. On considère un cycle géométrique $(V,f,g)$. Quitte à multiplier la métrique initiale $g$ par une constante, on peut supposer que $\sys(V,f,g)=2$ (ce qui ne change pas $\sigma(V,f,g)$), et \textbf{on fait cette hypothèse dans toute la suite}. On considère maintenant l'application  $\w f :\w V\to \w{K(\pi,1)}$, où $\w{K(\pi,1)}\to K(\pi,1)$ est le revêtement universel de $K(\pi,1)$. On note $\w V_0$ le relevé de $V_0$ à $\w V$.\\

Regardons quelques conséquences de l'hypothèse $\sys(V,f,g)=2$. Supposons que $\w v$ et $\w w$ dans $\w V$ vérifient $\w \dist(\w v,\w w)\le 1$. On a alors, pour tout $\gamma\neq1_\Gamma$ dans $\Gamma$ :
$$2\le \w\dist(\w v,\gamma\cdot \w v)\le \w\dist(\w v,\gamma\cdot \w w)+\w\dist(\gamma\cdot\w w, \gamma\cdot\w v),$$

Mais $\Gamma$ agit par isométries sur $\w V$ : $ \w\dist(\gamma\cdot\w w, \gamma\cdot\w v)=\w\dist(\w w,\w v)$. Il vient donc :
$$\w\dist(\w v,\gamma\cdot \w w)\ge 1.$$

Une conséquence est la suivante. Supposons que $\w v$, $\w w$ dans $\w V$ et $\alpha$ dans $\Gamma$ vérifient $\w \dist(\w v,\alpha\cdot\w w)\le 1$. Pour tout $\beta\neq1_\Gamma$ dans $\Gamma$, on a donc $\w \dist(\w v,\beta\cdot(\alpha\cdot\w w))\ge 1$. Prenons maintenant $\gamma\neq \alpha$ dans $\Gamma$. Avec  $\beta=\gamma\alpha^{-1}$ (qui est différent de $1_\Gamma$) on obtient :
$$\w \dist(\w v,\gamma\cdot\w w)\ge 1.$$

Pour $\psi\in L^\infty(\w V_0)$ et $\gamma\in \Gamma$, on considère l'élément $\gamma\cdot \psi$ de $L^\infty(\w V_0)$ défini, pour tout $\w w\in \w V_0$, par :
$$\gamma\cdot \psi(\w w)=\psi(\gamma^{-1}\cdot \w w).$$

On définit ainsi une action de $\Gamma$ sur $L^\infty(\w V_0)$.

Appelons maintenant $\w D$ un domaine fondamental (connexe) de l'action de $\Gamma$ sur $\w V$. Pour $\w v \in \w D$, on considère l'application :
$$\w I_0(\w v):\w V_0\to \R$$
définie par : 
%%%%%%%%%%%%%%%%%%%%%%%%%%%%%%
$I_0(\w v)(\w w)=\begin{cases}\min(1,\dist(p(\w v),p(\w w)))\ \mbox{si $\w w\in  \w D\cap\w V_0$};\\1\ \ \mbox{sinon}.\end{cases}$.
%%%%%%%%%%%%%%%%%%%%%%%%%%%%%%

Lorsque $\gamma\in \Gamma\setminus\set{1_\gamma}$, on pose encore : $\w I_0(\gamma\cdot \w v)=\gamma\cdot \w I_0(\w v)$, de sorte que $\w I_0:\w V\to L^\infty(\w V_0)$ est une application $\Gamma$-équivariante.\\ 

\begin{exem}
Illustrons cette construction avec un exemple élémentaire (voir figure \ref {fig-exempleI0}). On prend pour $V$ le cercle $\quot{[0,2]}{0\sim 2}$. On note $[x]$ la classe de $x$ dans $V$ et on prend $V_0=\set{[0],[\tfrac 23],[\tfrac 43]}$. On a ici $\Gamma=\Z$, $\w V=\R$ et $\w V_0=\tfrac 23\Z$. On choisit $\w D=[0,2[$. L'application $\w I_0(0)$ est définie par :
$$\w I_0(0)(x)=\begin{cases}0&\mbox{si $x=0$}\\\tfrac 23&\mbox{si $x=\tfrac 23$}\\ 1&\mbox{sinon}\end{cases}.$$
A titre d'exemple, par définition, pour $x$ réel, on a : $\w I_0(2)(x)=\w I_0(0)(x-2).$
%%%%%%%%%%%%%%%%%%%%%%%%%%%%%%
\begin{figure}[h]
\begin{center}
\includegraphics[width=0.7\linewidth]{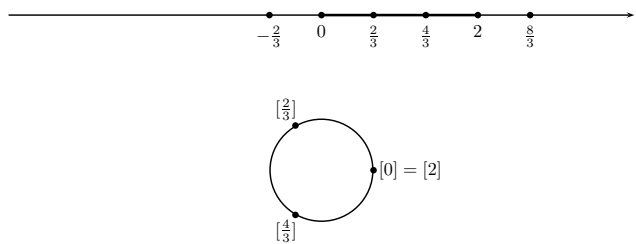}
\caption{Un exemple élémentaire}
\label{fig-exempleI0}
\end{center}
\end{figure}
%%%%%%%%%%%%%%%%%%%%%%%%%%%%%%
\end{exem}

Revenons à la construction générale. Prenons $v$ dans $V$ et considérons le cube $K(v)\subset L^\infty(V_0)$ (voir § \ref{subsection-complexecubique}). Pour $\varphi$ dans $K(v)$, on définit un élément $\w \varphi$ de $L^\infty(\w V_0)$ (de la même manière que pour $\w I_0$ ci-dessus) en posant, pour $\w w\in\w D\cap \w V_0$ et $\gamma$ dans $\Gamma\setminus\set{1_\gamma}$ : 
$$\w \varphi(\w w)=\varphi(p(\w w))\et \w \varphi(\gamma\cdot \w w)=1.$$

Lorsque $\w v$ est dans $\w D$ et vérifie $p(\w v)=v$, on note $\w K(\w v)$ l'ensemble des $\w \varphi$ dans $L^\infty(\w V_0)$ ainsi construites à partir des éléments $\varphi$ de $K(v)$, et, pour $\gamma$ dans $\Gamma$, on pose :
$$\w K(\gamma\cdot \w v)=\gamma\cdot \w K(\w v).$$

\begin{rema}
\label{rema-coordonnéescubegénéralisé}
Soit $\w v$ dans $\w V$. Pour $\w \varphi$ dans $\w K(\w v)$, il existe toujours deux éléments $\w w $ et $\w w'$ de $\w V_0$ tels que $\w\varphi(\w w)=0$ et $\w\varphi(\w w')=1$. En effet, l'application $\w \varphi$ est construite à partir  d'un élément $\varphi$ d'un certain cube $K(v)$ de $K(V)$, et on a vu (voir remarque \ref{rema-complexecubique}) que l'une des coordonnées de $\varphi$ est nulle.
\end{rema}

Notons maintenant $\w R_\eps$ la rétraction du complexe cubique $L^\infty(\w V_0)$ fournie par le théorème \ref{theo-retraction}. On peut construire, de la même manière que pour $V$, une extension $K(\w V)= K(\w V,\w V_0,\eps)$  de $\w V$, à partir de l'application $\w J=\w R_\eps\circ \w I_0$. 

Avec les notations ci-dessus, pour tout $\w v$ dans $\w V$, le cube minimal $K(\w v)$ de $L^\infty(\w V_0)$ qui contient $\w J(\w v)$ n'est autre que le cube $\w K(\w v)$. 
%%%%%%%%%%%%%%%%%%%%%%%%%%%%%%%%%%%%%%
%%%%%%%%%%%%%%%%%%%%%%%%%%%%%%%%%%%%%%
\begin{lemm} Le groupe $\Gamma$ agit sans point fixe et totalement discontinûment sur $K(\w V)$, de sorte que :
$$pr:K(\w V)\to \quot{K(\w V)}{\Gamma}=K(V)$$
est un revêtement régulier.
\end{lemm}
%%%%%%%%%%%%%%%%%%%%%%%%%%%%%%%%%%%%%%
%%%%%%%%%%%%%%%%%%%%%%%%%%%%%%%%%%%%%%
\begin{proof}[Démonstration]
Soit $\w \varphi$ dans $K(\w V)$. On considère la boule ouverte 
$$U=\set{\w \psi\in K(\w V)\mid\dist(\w \varphi-\w \psi)<\tfrac12}.$$
Prenons $\gamma$ dans $\Gamma\setminus\set{1_\Gamma}$. Raisonnons par l'absurde et supposons que $\gamma\cdot U \cap U\neq \emptyset$. Il existe alors $\w\psi\in U$ telle que $\norm{\w \varphi-\gamma\cdot\w \psi}_\infty<\tfrac12$ (voir remarque \ref{rema-topologiecomplexe} page \pageref{rema-topologiecomplexe}). Par inégalité triangulaire, on obtient :
\begin{equation}
\label{eq-norme}
	\norm{\w \psi-\gamma\cdot\w \psi}_\infty<1.
\end{equation}
La fonction $\w \psi$ appartient à un certain cube $K(\w v)$. Supposons que $\w v\in \w D$. Il existe $\w w$ dans $\w V_0\cap D$ tel que $\w \psi(\w w)=0$ (voir remarque \ref{rema-coordonnéescubegénéralisé}). Or $\w \psi(\gamma^{-1}\cdot\w w)=1$, de sorte que $\mod{\w \psi(\w w)-\gamma\cdot\w \psi(\w w)}=1$, ce qui contredit (\ref{eq-norme}). Si maintenant $\w v\notin \w D$, on dispose de $\alpha\in \Gamma$, de $\w v'$ dans $\w D$ et de $\w \psi_0$ dans $\w K(\w v')$ tels que $\w \psi=\alpha\cdot\w \psi_0$. L'inégalité (\ref{eq-norme}) implique alors, pour tout $\w w\in \w V_0$ :
$$\mod{\w \psi_0(\w w)-\w\psi_0(\gamma^{-1}\cdot\w w)}<1,$$
et on peut conclure à une absurdité de la même façon.
Ainsi l'action de $\Gamma$ sur  $K(\w V)$ est sans point fixe et totalement discontinue.

Enfin, pour $\w \varphi$ dans $K(\w V)$, la classe $pr(\varphi)$ est représentée par un élément 
$\w \varphi_0$ d'un certain cube $K(\w v)$ de $K(\w V)$ où $\w v$ est dans  $\w D$. Notons $v=p(\w v)\in V$. Alors $pr(\w \varphi_0)$ n'est autre que l'élément $\varphi_0$ de $K(v)$ défini, pour $w\in V_0$, par :
$$\varphi_0(w)=\w \varphi_0(\w w ),$$
où $\w w\in \w D\cap p^{-1}(w)$. On a ainsi $\quot{K(\w V)}{\Gamma}=K(V)$.
\end{proof}

On peut noter que l'application $\w J : \w V\to K(\w V)$ est encore lipschitzienne de rapport $\tfrac{1}{1-2\eps}$ et que le diagramme suivant est commutatif :
$$\begin{CD}
\w V @>\w J>>K(\w V) \\ @V p VV @VV pr V\\ V@>>J> K(V)
\end{CD}$$

%%%%%%%%%%%%%%%%%%%%%%%%%%%%%%%%%%%%%%%%%%%%%%%%%%%%%%%%%%%%%%%%%%%%
\begin{lemm}
\label{lemmecomplexe}
Soient $\w v$ et $\w v'$ des points dans $\w V$. On note $K_1=K(\w v)$ et $K_2=K(\w v')$.\\
Si pour  $m\ge 1$ entier on a $\w \dist(\w v,\w v')\ge m$ alors :
$$\dist(K_1,K_2)\ge m.$$
\end{lemm}
%%%%%%%%%%%%%%%%%%%%%%%%%%%%%%%%%%%%%%%%%%%%%%%%%%%%%%%%%%%%%%%%%%%%

\begin{proof}[Démonstration]
Rappelons que $K_1$ et $K_2$ sont les cubes minimaux dans $K(\w V)$ qui contiennent respectivement $\w J(\w v)$ et $\w J(\w v')$.\\
On procède par récurrence sur $m\ge 1$. Montrons le résultat pour $m=1$. On suppose que $\w \dist(\w v,\w v')\ge 1$. On choisit un point $\w w$ de $\w V_0$ suffisamment proche de $\w v$, de sorte que $\w J(\w v)(\w w)=0$ et $\w J(\w v')(\w w)=1$. Par minimalité de $K_1$ et $K_2$, pour tout $\w \varphi$ dans $K_1$ et tout $\w \psi$ dans $K_2$, on a $\w \varphi(\w w)=0$ et $\psi(\w w)=1$, donc
$$1=\norm{\w \psi-\w \varphi}_\infty\le \dist(\w \psi,\w \varphi).$$
Il en résulte que $\dist(K_1,K_2)\ge 1$.

Supposons maintenant le résultat vrai pour un certain entier naturel $m$. Soient $\w v$ et $\w v'$ dans $\w V$ tels que $\w \dist(\w v,\w v')\ge (m+1)$. 
On prend un plus court chemin dans $K(\w V)$ entre les cubes $K_1$ et $K_2$ et on prend un point $\w x$ sur ce chemin pour lequel $\dist(\w x,K_1)=1$. Ce point $\w x$, par construction de $K(\w V)$, appartient à  un certain cube minimal $K(\w w)$ pour un certain point $\w w$ de $\w V$. Comme $1= \dist (K_1,\w x)\le \dist (\w J(\w v),\w x)$ on a :

\begin{equation}
\label{eqcomplexe}
	\dist(K_1,K_2)= \dist(K_1,\w x)+\dist(\w x,K_2) \ge 1+\dist(K(\w w),K_2).
\end{equation}
De plus $\w \dist(\w w,\w v)\le 1$ (puisque $\dist(K_1,K(\w w))\le 1$). Ainsi on a $$\w \dist(\w w,\w v')\ge \w \dist(\w v,\w v')-\w \dist(\w w,\w v)\ge \w \dist(\w v,\w v')-1=m.$$ Le résultat étant supposé vrai au rang $m$, on  a 
$\dist(K_2,K(\w w))\ge m$. 
On peut donc conclure, via (\ref{eqcomplexe}), que le résultat est  vrai pour $m+1$.
\end{proof}

%%%%%%%%%%%%%%%%%%%%%%%%%%%%%%%%%%%%%%%%%%%%%%%%%%%%%%%%%%%%%%%%%%%
\begin{lemm}
\label{prolongement}
L'application $f':V'\to \K$ (voir \normalfont{\ref{rem-essentielle}}) se prolonge en une application continue $F:K(V)\to\K$.
\end{lemm}

%%%%%%%%%%%%%%%%%%%%%%%%%%%%%%%%%%%%%%%%%%%%%%%%%%%%%%%%%%%%%%%%%%%%
\begin{proof}[Démonstration]
Soit $\w{V'}$ le relevé de $V'$ à $K(\w V)$. Notons  $\w f':\w{V'}\to \w{\K}$ l'application équivariante qui relève $f':V'\to \K$ :
$$\begin{CD}
\w {V'} @>\w {f'}>>\w {\K} \\ @V pr\mid_{\w {V'}} VV @VV q V\\ V'@>>f'> \K
\end{CD}$$
Soit $\Delta$ une partie de $K(\w V)$ qui contient exactement un représentant de chaque orbite de l'action de $\Gamma$ sur $K(\w V)$, et qui est d'intersection  non vide avec $\w {V'}$. Puisque $\w {\K}$ est simplement connexe, la restriction de $\w f'$ à $\Delta\cap\w{V'}$, se prolonge en une application continue $\w {f'}_1:\Delta\cup \w{V'}\to \w{\K}$.
On prolonge alors cette application $\w{f'}$ en une application équivariante 
$$\w F:K(\w V)\to \w {\K},$$
en posant, $\w F(\gamma\cdot \w v')=\gamma\cdot  \w {f'}_1(\w v')$. Cette application $\w F$ induit alors une application continue 
$$F:K(V)\to \K$$
telle que le diagramme suivant commute, où $q:\w \K\to \K$ est la projection canonique.
$$\begin{CD}
 {K(\w V)} @>\w F>>\w {K(\pi,1)} \\ @V pr VV @VV q V\\ K(V)@>>F> \K
\end{CD}$$
Cette dernière application $F$ prolonge $f'$. 
\end{proof}

%%%%%%%%%%%%%%%%%%%%%%%%%%%%%%%%%%%%%%%%%%%%%%%%
\label{notations}
%%%%%%%%%%%%%%%%%%%%

Soit maintenant $V''$ une pseudo-variété de dimension $n$ plongée dans $K(V)$. Sa classe fondamentale est représentée par un cycle $c_{V''}=\dsum_{i=1}^m \sigma_i$ où $\set{\sigma_1,\ldots,\sigma_m}$ est l'ensemble des simplexes de dimension $n$ qui représentent $V''$. Ce cycle peut aussi bien se considérer comme une chaîne singulière de $K(V)$, qui est clairement un cycle. On confondra par la suite $c_{V''}$ et $V''$.

\begin{defi} On dira que les pseudo-variétés $V'=J(V)$ et $V''$ sont \textit{homologues dans $K(V)$}\index{Pseudo-variétés homologues} lorsque les cycles $c_{V'}$ et $c_{V''}$ représentent la même classe dans $H_n(K(V);\Z)$.
\end{defi}

On munit $V''$ de la métrique riemannienne $g''$ fournie par la remarque \ref{rema-metrique}, page \pageref{rema-metrique}. On a vu que $(V', f', g')$ est un cycle géométrique qui représente la classe $h$ dans $H_n(\pi;\Z)$ et que l'application $f':V'\to \K$ se prolonge en une application continue $F :K(V)\to \K$. Notons $f''$ la restriction de $F$ à $V''$. \\
 
%%%%%%%%%%%%%%%%%%%%%%%%%%%%%%%%%%%%%%%%%%%%%%

\begin{theo}
\label{theosyst}
Soit  $V''$ une pseudo-variété, homologue à $V'$ dans $K(V)$. Le cycle géométrique $(V'',f'',g'')$ représente la classe $h$. Si de plus le cycle $(V,f,g)$ est normalisé, i.e. lorsque $\im f_*=\pi$, où $f_*:\pi_1(V)\to \pi$, alors :
$$\sys(V'',f'',g'')\ge \sys(V,f,g),$$
\end{theo}
%%%%%%%%%%%%%%%%%%%%%%%%%%%%%%%%%%%%%%%%%%%%%

\begin{proof}[Démonstration]
Rappelons que l'on a supposé $\sys(V,f,g)=2$.
Notons $f'_*:H_n(V';\Z)\to H_n(\pi;\Z)$, $f''_*:H_n(V'';\Z)\to H_n(\pi;\Z)$ et $F_*:H_n(K(V);\Z)\to H_n(\pi;\Z)$ les morphismes induits au niveau des groupes d'homologie. On a alors :
$$f''_*[V'']=F_*[c_{V''}]=F_*[c_{V'}]=f'_*[V']=h.$$

En effet, montrons que $F_*[c_{V''}]=f''_*[V'']$ (la seconde égalité résultant du fait que $V''$ est homologue à $V'$ dans $K(V)$). Pour cela, regardons ce qui se passe au niveau des complexes de chaînes. Rappelons  que le cycle $c_{V''}$, qui représente la classe fondamentale $[V'']$ dans $H_n(V'';\Z)$, peut être considéré comme un cycle de $K(V)$. Nous mettrons $\sharp$ en indice pour les applications induites au niveau des complexes de chaînes. Désignons par $\mathcal C$ une chaîne (cycle) de dimension $n$ dans $K(\pi,1)$ qui représente la classe $h''=f''_*[V'']$. Il existe $\mathcal C'$ chaîne de dimension $n+1$ dans $\K$ telle que $f''_\sharp(c_{V''})=\mathcal C+\partial \mathcal C'$. Comme $F$ prolonge $f''$ à $K(V)$, on a $F_\sharp(c_{V''})=f''_\sharp(c_{V''})$ et ainsi  : 
$$F_\sharp(c_{V''})=\mathcal C+\partial \mathcal C'.$$
Il en résulte que $F_\sharp(c_{V'})$ est homologue à $\mathcal C$ donc $F_*[c_{V''}]=f''_*[V'']$.\\

Passons à l'inégalité concernant les systoles relatives. On se rappelle tout d'abord que 
$$\sys(V'')=\inf \left\{\w\dist(\w v'',\gamma\cdot \w v'')\mid \mbox{$\w v''\in \w {V''}$ et $\gamma\in G(\w {V''})\setminus\left\{1_{G(\w {V''})}\right\}$}\right\},$$
où $\w {V''}\to V''$ est le revêtement associé au sous-groupe $\ker f''_*$  de $\pi_1(V'')$ et $G(\w {V''})$ est le groupe des automorphismes de ce revêtement. Soit  $\w v''\in \w V''\subset K(\w V)$, cette dernière inclusion résultant du fait que $(V,f,g)$ est un cycle normalisé. Pour $\gamma$ dans $G(\w {V''})\setminus \left\{1_{G(\w {V''})}\right\}$,  on choisit un cube $K_1$ de $K(\w V)$ qui contient $\w v''$. 
Par construction de $K(\w V)$, il existe $\w v$ dans $\w V$ tel que $K_1$ contienne $\w J(\w v)$. Le cube $K_2=\gamma\cdot K_1$ contient alors $\gamma\cdot \w v''$, ainsi que $\gamma \cdot \w J(\w v)=\w J(\gamma\cdot \w v)$. Il vient ainsi :
$$\w \dist(\w v'',\gamma\cdot \w v'')\ge \dist(K_1,K_2).$$
Comme $\sys(V,f,g)=2 \le \w \dist(\w v,\gamma\cdot \w v)$, on obtient, d'après le lemme \ref{lemmecomplexe} appliqué à $m=2$, $\dist(K_1,K_2)\ge \sys (V,f,g)$. Ainsi :
$$\w \dist(\w v'',\gamma\cdot \w v'')\ge\sys (V,f,g).$$ 
Cette dernière inégalité étant valable pour tout $\w v''$ dans $\w {V''}$, on obtient le résultat souhaité, via le théorème \ref{theo-defsystole}.
\end{proof}

\begin{rema} On peut toujours supposer, et ce sera le cas dans la suite, que le cycle géométrique initial $(V,f,g)$ est normalisé : voir \cite{Babenko-TSU}.
\end{rema}

%%%%%%%%%%%%%%%%%%%%%%%%%%%%%%%%%%%%%%%%%%%%%%%%%%%%%%%%%%%%%%%%%%%%
%%%%%%%%%%%%%%%%%%%%%%%%%%%%%%%%%%%%%%%%%%%%%%%%%%%%%%%%%%%%%%%%%%%%
 \section{Preuve des théorèmes A et B}
%%%%%%%%%%%%%%%%%%%%%%%%%%%%%%%%%%%%%%%%%%%%%%%%%%%%%%%%%%%%%%%%%%%%%%%%%

%%%%%%%%%%%%%%%%%%%%%%%%%%%%%%%%%%%%%%%%%%%%%%%%%%%%%%%%%%%%%%%%%
\subsection{La preuve du théorème B}
%%%%%%%%%%%%%%%%%%%%%%%%%%%%%%%%%%%%%%%%%%%%%%%%%%%%%%%%%%%%%%%%%%
On part d'un cycle géométrique $(V,f,g)$ (normalisé, voir remarque ci-dessus et théorème \ref{theosyst}) qui représente $h$ tel que $\sys(V,f,g)=2$. On choisit une partie $V_0$ de $V$, suffisamment dense dans $V$, ce qui permet de construire, pour $\eps>0$ suffisamment petit, le complexe cubique $K(V)$ ainsi que la pseudo-variété $V'=J(V)$ de $K(V)$ (voir la remarque \ref{rem-essentielle}).\\

Introduisons maintenant une notion particulière de volume d'une classe d'homologie $a$ dans $H_n(K(V);\Z)$. Pour $a$ dans $H_n(K(V);\Z)$, on considère l'ensemble $C(a)$ des pseudo-variétés de dimension $n$ incluses dans $K(V)$ qui représentent $a$. On pose :
$$\vol(a)=\inf \left\{\vol(c)\mid c\in C(a)\right\}.$$ 

Notons, par abus, $[V']$ la classe d'homologie de $c_{V'}$ dans $K(V)$. On appellera \textit{suite minimisante}\index{Suite minimisante} (pour $\vol[V']$) une suite $c_i$  de pseudo-variétés dans $K(V)$, homologues à $V'$, telle que :
$$\vol(c_i)\underset{+\infty}{\longrightarrow}\vol[V'].$$
Si une suite minimisante $(c_i)$ converge au sens de Hausdorff vers une partie compacte $W$ de $K(V)$, on dira que $W$ est \textbf{\textit{minimale}} lorsqu'aucune partie propre de $W$ (i.e. aucune partie incluse dans $W$ et distincte de $W$) n'est limite de Hausdorff d'une autre suite minimisante $c_i'$ dans $K(V)$.\\

Rappelons aussi ce qu'est la distance de Hausdorff. Soit $(X,\dist)$ un espace métrique. Notons $\mathcal K(X)$ l'ensemble de toutes les parties compactes de $X$. Pour $A\subset X$ et $\eps>0$, on pose :
$$\text{tube}(A,\eps)=\left\{x\in X\mid \dist(x,A)<\eps\right\}$$
Pour $A$ et $B$ dans $\mathcal K(X)$, la distance de Hausdorff\index{Distance de Hausdorff} entre $A$ et $B$ est :
$$\dist_{\mathcal H}(A,B)=\inf \left\{\eps>0\mid A\subset \text{tube}(B,\eps)\ \ \mbox{et}\ \ B\subset \text{tube}(A,\eps)\right\}$$\\
On sait alors que $(\mathcal K(X),\dist_{\mathcal H})$ est un espace métrique, qui de plus est compact lorsque $X$ l'est. 
On utilisera dans la suite les résultats suivants (voir \cite{BBI}).
\begin{quote}
\noindent \textbf{R1.} \label{H1} \textit{Si $(A_p)$ est une suite dans $\mathcal K(X)$ telle que $A_p\tend A\in \mathcal K(X)$, pour $\dist_{\mathcal H}$, alors  $A$ est exactement l'ensemble des limites des suites convergentes $(x_p)$ dans $X$ telles que $x_p\in A_p$ pour tout $p\in \N$}.

\noindent \textbf{R2.} \label{H2} \textit{Si  $(A_p)$ est une suite décroissante (pour l'inclusion) dans $\mathcal K(X)$, alors elle converge  pour $\dist_{\mathcal H}$ vers $\displaystyle \bigcap_{p\in \N}A_p$ dans $\mathcal K(X)$}.
\end{quote}

%%%%%%%%%%%%%%%%%%%%%%%%%%%%%%%%%%%%%%%%%%%%%%%%%%%%%%%%%%%%%%%%%%%%%%%%%%%%%%%%%%%
\begin{lemm}
\label{partieminimale}
On  peut choisir une suite minimisante $(V_i)$ de pseudo-variétés dans $K(V)$ qui converge pour la distance de Hausdorff vers une partie compacte minimale, notée $W_\infty$, de $K(V)$.
\end{lemm}
%%%%%%%%%%%%%%%%%%%%%%%%%%%%%%%%%%%%%%%%%%%%%%%%%%%%%%%%%%%%%%%%%%%%%%%%%%%%%%%%%%%

\begin{proof}[Démonstration]
Soit $\mathcal W$ l'ensemble des parties compactes de $K(V)$ qui sont limites de Hausdorff de suites minimisantes, muni de la relation d'ordre partielle $\subset$. On montre que c'est un ensemble inductif au sens suivant : toute partie totalement ordonnée de $\mathcal W$ admet un minorant dans $\mathcal W$. Le lemme de Zorn fournira alors un élément minimal $W_\infty$ de  $\mathcal W$.

Soit $\mathcal W_0$ une partie totalement ordonnée de $\mathcal W$. On écrit $\mathcal W_0=(W_\alpha)_{\alpha\in A}$. On considère :
$$K_\infty=\displaystyle\bigcap_{\alpha\in A}W_\alpha.$$

C'est une partie compacte de $K(V)$. On montre alors qu'il existe une suite $(W_i)_{i\in \N}$  dans $\mathcal W_0$, décroissante pour l'inclusion, telle  que :  
$$K_\infty=\displaystyle\bigcap_{i\in \N}W_i.$$

En effet, pour $i$ dans $\N$, notons $U_i$ le voisinage tubulaire ouvert de $K_\infty$ de rayon $2^{-i}$ dans $K(V)$. Fixons $i$ dans $\N$. Le complémentaire $F_i$ de $U_i$ dans $K(V)$ est un fermé de $K(V)$ : c'est donc un compact. De plus, comme $K_\infty\subset F_i$, on a :
$$F_i\subset \displaystyle\bigcup_{\alpha\in A}O_\alpha,$$
où $O_\alpha$ est le complémentaire de $W_\alpha$ dans $K(V)$. Par compacité, il existe une partie finie $B_i$ de $A$ telle que :
$$F_i\subset\displaystyle\bigcup_{\beta\in B_i}O_\beta.$$
Ainsi, $\displaystyle\bigcap_{\beta\in B_i}W_\beta\subset U_i$. 

Comme $B_i$ est une partie finie de $A$,  l'ensemble $\set{W_\beta}_{\beta\in B_i}$ admet un plus petit élément, que l'on note $W_i$. On a alors $W_i\in \mathcal W_0$ et $K_\infty\subset W_i\subset U_i$. 

Il en résulte que : $K_\infty=\displaystyle\bigcap_{i\in \N}W_i$. Comme on peut de plus imposer $B_i\subset B_{i+1}$ pour tout $i\in \N$, la suite $(W_i)_{i\in \N}$ est décroissante.

Montrons maintenant que $K_\infty\in \mathcal W$. Pour $j\in \N$, chaque $W_j$ est dans $\mathcal W$ : il existe donc une suite minimisante $(c_i^j)_{i\in \N}$ telle que :
$$c_i^j\underset{i\to +\infty}{\overset{H}{\longrightarrow}}W_j.$$

Pour tout $j$ dans $\N$, il existe $k_j\in \N$ tel que pour tout $i\ge k_j$ on ait :
$$\mod{\vol(c_i^j)-\vol[V']}\le 2^{-j}.$$

Puis, comme $c_i^j\overset{H}{\longrightarrow} W_j$, il existe $\ell_j\in \N$ tel que, pour tout $i\ge \ell_j$, on ait :

\begin{equation}
\label{eqlemmin}
	\dist_{\mathcal H}(c_i^j,W_j)\le 2^{-j}
\end{equation}

On construit alors une suite minimisante $(c'_j)$ en posant, pour tout $j$ dans $\N$ : $c'_j=c_{i_j}^j$ où $i_j=\max(k_j,\ell_j)$. 
Il s'agit bien d'une suite minimisante puisque pour tout $j$ entier naturel on a :
$$\mod{\vol(c'_j)-\vol[V']}\le 2^{-j}.$$

Puis, on sait que $W_i\overset{H}{\longrightarrow}K_\infty$ (c'est la propriété \textbf{R2}). Ainsi, pour tout $\eps>0$, il existe $j_0$ entier tel que, pour tout $j\ge j_0$, on ait :
$$\dist_{\mathcal H}(W_j,K_\infty)\le \eps.$$

D'après (\ref{eqlemmin}) on a : $\dist_{\mathcal H}(c_j',W_j)\le 2^{-j}$. Pour $j\ge j_1$ convenable, on a donc :
$$\dist_{\mathcal H}(c_j',K_\infty)\le\eps.$$
Ainsi, pour $j\ge \max(j_0,j_1)$, il vient, par inégalité triangulaire :
$$\dist_{\mathcal H}(c_j',K_\infty)\le\dist_{\mathcal H}(c_j',K_\infty)+\dist_{\mathcal H}(K_\infty,W_j)\le 2\eps$$
ce qui prouve que la suite $(c_j')$ converge au sens de Hausdorff vers $K_\infty$. Il en résulte que $K_\infty$ est un minorant dans $\mathcal W$ de $\mathcal W_0$. 
\end{proof}

On peut noter qu'il n'y a pas forcément unicité de la partie minimale compacte $W_\infty$ de $K(V)$. On aura besoin, pour la suite, des deux résultats suivants.

%%%%%%%%%%%%%%%%%%%%%%%%%%%%%%%%%%%%%%%%%%%%%%%%%%%%%%%%%%%%%%%%%%%%%%%%%%%%%%%%%%%%%
%%%%%%%%%%%%%%%%%%%%%%%%%%%%%%%%%%%%%%%%%%%%%%%%%%%%%%%%%%%%%%%%%%%%%%%%%%%%%%%%%%%%%
\begin{lemm}
\label{lem-analytique}
Soient $0\le \alpha\le \beta$ et $c>0$ des réels, ainsi que $a:[\alpha,\beta]\to \R_+^*$ et $v:[\alpha,\beta]\to \R^+$ des fonctions. On suppose que :
\begin{quote}
{H1.} Pour tout $R\in [\alpha,\beta]$ on a $v(R)\ge \dint_\alpha^R a(t)\:\text{d}t$;\\
{H2.} Pour tout $R\in [\alpha,\beta]$ on a $v(R)\le c[a(R)]^{\frac{n}{n-1}}$.
\end {quote}
Alors, pour tout $R$ dans $[\alpha,\beta]$, on a : $v(R)\ge \dfrac{1}{c^{n-1}n^n}(R-\alpha)^n$.
\end{lemm}
%%%%%%%%%%%%%%%%%%%%%%%%%%%%%%%%%%%%%%%%%%%%%%%%%%%%%%%%%%%%%%%%%%%%%%%%%%%%%%%%%%%%%

\begin{proof}[Démonstration]
Pour tout $R$ dans $[\alpha,\beta]$, on pose $A(R)=\dint_\alpha^R a(t)\:\text{d}t$. Pour tout $t$ dans $[\alpha,\beta]$, selon  \textit{H1} et \textit{H2}, on a  : $A(t)\le  c[A'(t)]^{\frac{n}{n-1}}$, ce qui amène :
$$\dfrac{A'(t)}{A(t)^{\frac{n-1}{n}}}\ge\left( \dfrac{1}{c}\right)^{\frac{n-1}{n}}.$$ 
En intégrant, pour $R$ dans $[\alpha,\beta]$, on obtient : $$n \left[A(R)^{\frac{1}{n}}-A(\alpha)^{\frac{1}{n}}\right]\ge \left( \dfrac{1}{c}\right)^{\frac{n-1}{n}}(R-\alpha).$$
%$$\int_\alpha^x\dfrac{A'(t)}{A(t)^{\frac{n^-1}{n}}}\:\text{d}x\ge \dfrac{1}{c^{\frac{n^-1}{n}}}(x-\alpha),$$
Comme $A(\alpha)\ge 0$, il vient, pour tout $R\in [\alpha,\beta]$ : $A(R)\ge \dfrac 1{n^nc^{n-1}}(R-\alpha)^n$.

On utilise  alors \textit{{H1}} pour conclure.
 \end{proof}
%%%%%%%%%%%%%%%%%%%%%%%%%%%%%%%%%%%%%%%%%%%%%%%%%%%%%%%%%%%%%%%%%%%%%%%%%%%%%%%%%%%%%

\begin{theo} [Inégalité d'Eilenberg]\index{Inégalité d'Eilenberg} \textit{Soient $X$ un espace métrique et $z$ un cycle de dimension $n$ dans $X$. Il existe une constante $E_n$, qui ne dépend que de la dimension $n$, telle que pour toute  application 1-lipschitzienne $f:X\to \R$ on ait :
$$\vol(z\cap\left\{f\le r\right\}) \ge \dint_0^r\vol(z\cap\left\{f=t\right\})\d t,$$
pour presque tout $r$ dans $\R$. De plus, 
$z\cap\left\{f=r\right\}$ est un cycle de dimension $n-1$ pour presque tout $r\in \R$.}
\end{theo}

%%%%%%%%%%%%%%%%%%%%%%%%%%%%%%%%%%%%%%%%%%%%%%%%%%%%%%%%%%%%%%%%%%
%%%%%%%%%   MODIF
%%%%%%%%%%%%%%%%%%%%%%%%%%%%%%%%%%%%%%%%%%%%%%%%%%%%%%%%
Pour ce résultat, je renvoie à \cite{Gromov-FRM} page 21 et \cite{BZ} page 101. On va utiliser ce résultat avec $X=K(V)$ et la fonction $f=\dist(\ast,.)$ où $\ast$ est une partie de $K(V)$. L'hypothèse {\textit{H1}} du lemme \ref{lem-analytique} est en fait l'inégalité d'Eilenberg.\\

\label{defW0}
Le lemme  \ref{partieminimale} nous autorise à prendre une suite minimisante $(V_i)$ (pour $\vol[V']$) de pseudo-variétés dans $K(V)$, qui converge pour la distance de Hausdorff vers une partie compacte minimale $W_\infty$ de $K(V)$.

%%%%%%%%%%%%%%%%%%%%%%%%%%%%%%%%%%%%%%%%%%%%%%%%%%%%%%%%%%%%%%%%%%
%%%%%%%%%   MODIF
%%%%%%%%%%%%%%%%%%%%%%%%%%%%%%%%%%%%%%%%%%%%%%%%%%%%%%%%

Fixons maintenant $a\in ]0,1[$ et $\rho>0$. On va alors démontrer le résultat suivant.\\
%%%%%%%%%%%%%%%%%%%%%%%%%%%%%%%%%%%%%%%%%%%%%%%%%%%%%%%%%%%%%%%%%%%%%%%%%%%%%%%%%%%%%
%%%%%%%%%%%%%%%%%%%%%%%%%%%%%%%%%%%%%%%%%%%%%%%%%%%%%%%%%%%%%%%%%%%%%%%%%%%%%%%%%%%%
%%%%%%%%%%%%%%%%%%%%%%%%%%%%%%%%%%%%%%%%%%%%%%%%%%%%%%%%%%%%%%%%%%%%%%%%%%%%%%%%%%%%%
%%%%%%%%%%%%%%%%%%%%%%%%%%%%%%%%%%%%%%%%%%%%%%%%%%%%%%%%%%%%%%%%%%%%%%%%%%%%%%%%%%%%
\begin{lemm}
\label{lemmevoisinagetub}
Il existe une constante réelle $A_n>0$ (indépendante de $a$ et $\rho$), et une extraction $\varphi:\N \to \N$  telle que, pour tout $i$ dans $\N$, pour tout $v$ dans $V_{\varphi(i)}$ et tout $R$ dans $[a,1]$ on ait :
$$\vol(c_{\varphi(i)}(R))\ge A_n (R-a)^n$$
où $c_i(R)=V_i\cap \tub(B_i(v,R),\rho)$, $\tub(B_i(v,R),\rho)$ désignant le voisinage tubulaire de rayon $\rho$ de la boule $B_i(v,R)$ de $V_i$ dans $K(V)$.\\
\end{lemm}
%%%%%%%%%%%%%%%%%%%%%%%%%%%%%%%%%%%%%%%%%%%%%%%%%%%%%%%%%%%%%%%%%%%%%%%%%%%%%%%%%%%%%
%%%%%%%%%%%%%%%%%%%%%%%%%%%%%%%%%%%%%%%%%%%%%%%%%%%%%%%%%%%%%%%%%%%%%%%%%%%%%%%%%%%%
%%%%%%%%%%%%%%%%%%%%%%%%%%%%%%%%%%%%%%%%%%%%%%%%%%%%%%%%%%%%%%%%%%%%%%%%%%%%%%%%%%%%%
%%%%%%%%%%%%%%%%%%%%%%%%%%%%%%%%%%%%%%%%%%%%%%%%%%%%%%%%%%%%%%%%%%%%%%%%%%%%%%%%%%%%
On raisonne par l'absurde. On suppose donc que :  
\begin{quote}
\label{propriété-P}
$\underline{(\mathcal P)}$ \textit{Pour tout réel $A>0$, il existe $i_0\in \N$ tel que, pour tout $i\ge i_0$, on ait l'existence de $v_i$ dans $V_i$ et de $R_i\in[a,1]$ pour lesquels :
$$\vol(c_i(R_i))< A(R_i-a)^n$$
où $c_i(R_i)=V_i\cap \tub(B_i(v,R_i),\rho)$}. \\
\end{quote}

Pour aboutir à une contradiction, on va construire  une suite minimisante de cycles pour le volume de la classe d'homologie $[V']$ qui va converger vers une partie propre de $W_\infty$, ce qui contredira le caractère minimal de $W_\infty$.

Pour simplifier la lecture de ce qui suit, notons, pour tout $t\ge 0$ et $i\ge i_0$ dans $\N$ :

%%%%%%%%%%%%%%%%%%%%%%%%%%%%%%%%%%%%%%%%%%%%%%%%%%%%%%%%%%%%%%%%%%
%%%%%%%%%   MODIF

$$c_i(t)=V_i\cap \tub(B_i(v_i,t),\rho)\ \ \mbox{et}\ \ z_i(t)=\partial c_i(t).$$

Fixons provisoirement $A>0$. Comme la propriété $(\mathcal P)$ est supposée vraie, il existe une suite $(r_i)$ dans $[a,R_i]$ telle que :
\begin{equation}
	\label{eq10}
	\vol(c_i(r_i))> \left(\dfrac{1}{An^n}\right)^{\tfrac 1{n-1}}\vol(z_i(r_i))^{\frac n{n-1}}
\end{equation}
où $z_i(r_i)=\partial c_i(r_i)$. 
En effet, cela résulte du lemme \ref{lem-analytique}, appliqué aux réels $\alpha=a$, $\beta=R_i$ et aux fonctions $a(t)=\vol(z_i(t))$ et $v(t)=\vol(c_i(t))$.

Il vient alors :
\eq
\vol(z_i(r_i))&<&\left(\left(An^n\right)^{\tfrac 1{n-1}}\vol(c_i(r_i))\right)^{\frac{n-1}n}\\
&\le& A^{\tfrac 1n}n\:\vol(c_i(r_i))^{\frac {n-1}{n}}\le  A^{\tfrac 1n}n\:\vol(c_i(R_i))^{\frac {n-1}{n}}\\
&\le& A^{\tfrac 1n}n\:A^{\frac {n-1}{n}}R_i^{\frac {n-1}{n}}\le An.
\fineq

Supposons maintenant que  :
\begin{equation}
\label{eq-A1}	A n\le \alpha_{n-1},
\end{equation}
où $\alpha_{n-1}$ est une des deux constantes définies par le théorème \ref{theo-inégalitéisocomplexes} (voir page \pageref{theo-inégalitéisocomplexes}).

En vertu de ce théorème, il existe, pour tout $i\ge i_0$, une chaîne $C_i$ de dimension $n$ dans $K(V)$, dont le bord est $z_i(r_i)$, qui est contenue dans le voisinage tubulaire de rayon $\eps_i=\beta_{n-1}\:\vol(z_i(r_i))^{\frac 1{n-1}}$ et qui vérifie :

\begin{equation}
\label{1trefle}
	\vol(C_i)\le \beta_{n-1} \vol(z_i(r_i))^{\frac{n}{n-1}} 
\end{equation}

$\ $

On considère maintenant la suite de cycles : $$V_i'=V_i-c_i(r_i)+C_i\ \ \ (i\ge i_0).$$

On va démontrer qu'il s'agit d'une suite minimisante pour le volume de la classe d'homologie $[V']$ qui (à une extraction près) converge (pour la distance de Hausdorff) vers une partie propre de $W_\infty$. 

Tout d'abord, pour $i\ge i_0$ et pour $A$ petit, les $V_i'$ sont homologues à $V'$. En effet, pour tout $i\ge i_0$, $\partial c_i(r_i)=\partial C_i=z_i(r_i)$ donc $C_i-c_i(r_i)$ est un cycle de dimension $n$ dans $K(V)$. Puis
on a pour $i\ge i_0$, en vertu de la propriété ${(\mathcal P)}$ page \pageref{propriété-P}, $\vol(c_i(r_i))\le A$. Ainsi, pour $i\ge i_0$ : 
\eq
\vol (C_i-c_i(r_i))&\le&\vol(c_i(r_i))+\vol(C_i)\le A+\beta_{n-1}\vol (z_i(r_i))^{\frac n{n-1}}\\
&\le &A+\beta_{n-1}\left(An^n\right)^{\tfrac 1{n-1}}\vol(c_i(r_i))\le A+\beta_{n-1}(An)^{\frac n{n-1}}.
\fineq
Pour 

\begin{equation}
\label{eq-A2}
	A+\beta_{n-1}(An)^{\frac n{n-1}}\le \alpha_n,
\end{equation}
 $C_i-c_i(r_i)$ est un bord selon le théorème \ref{theo-inégalitéisocomplexes}.
%%%%%%%%%%%%%%%%%%%%%%%%%%%%%%%%%%%%%%%%%
\begin{lemm}
Pour un choix convenable de la constante $A$, on a : $$\underset{i\to +\infty}{\lim}\vol(V_i')=\vol[V'].$$
\end{lemm}
%%%%%%%%%%%%%%%%%%%%%%%%%%%%%%%%%%%%%%%%%

\begin{proof}[Démonstration] Montrons tout d'abord que :$$\vol(c_i(r_i))-\volremp\big(z_i(r_i)\subset K(V)\big)\underset{i\to +\infty}{\longrightarrow}0.$$
 On raisonne par l'absurde. Supposons que, quitte à extraire, on ait, pour tout $i\ge i_0$, $\vol(c_i)-\text{Vol\:Remp}\big(z_i(r_i)\subset K(V)\big)\ge \gamma>0$. 
On considère la suite de cycles $V_i''=V_i-c_i(r_i)+W_i$, où $W_i$ est une chaîne de $K(V)$ telle que $\vol(W_i)\le \text{Vol\:Remp}\big(z_i(r_i)\subset K(V)\big)+\tfrac{\gamma}{2}$ et $\partial W_i=z_i$. On peut d'ailleurs supposer que $\vol(W_i)\le \vol(c_i(r_i))$

Pour $i\ge i_0$, chaque $V_i''$ est encore homologue à $V'$, à condition que $A$ soit suffisamment petit. En effet, pour $i\ge i_0$, $W_i-c_i(r_i)$ est un cycle de dimension $n$ dans $K(V)$, et on a (grâce à la propriété ${(\mathcal P)}$): 
$$
\vol(W_i-c_i(r_i))\le \vol(c_i(r_i))+\vol(W_i)\le 2A.
$$
Ainsi, toujours par le théorème \ref{theo-inégalitéisocomplexes}, le cycle $W_i-c_i(r_i)$ est un bord dès que :

\begin{equation}
\label{eq-A3}
A\le \dfrac{\alpha_n}2.
\end{equation}
Or
\begin{eqnarray*}
\vol(V_i'')&\le& \vol(V_i-c_i(r_i))+\vol(W_i)\\
&\le&\vol(V_i)-\vol (c_i(r_i))+\volremp\big(z_i(r_i)\subset K(V)\big)+\tfrac{\gamma}{2}\\
&\le&\vol(V_i)-\gamma+\tfrac{\gamma}{2}\le\vol(V_i)-\tfrac{\gamma}{2}.
\end{eqnarray*}
Il en résulte que, pour $i$ grand, $\vol (V_i'')<\vol [V']$, ce qui contredit la définition de $\vol[V']$. On a donc  :

\begin{equation}
\label{2trefle}
	\vol(c_i(r_i))-\volremp\big(z_i(r_i)\subset K(V)\big)\underset{i\to +\infty}{\longrightarrow}0
\end{equation}

Puis, pour tout $i\ge i_0$, on a :
\eq
\lefteqn{\vol(c_i(r_i))-\volremp\big(z_i(r_i)\subset K(V)\big)\ge}\\&& \dfrac{1}{(An^n)^{\tfrac 1{n-1}}}\:\vol(z_i(r_i))^{\frac n{n-1}}-\beta_{n-1} \vol(z_i(r_i))^{\frac{n}{n-1}}\\
&\ge &\left(\left(\dfrac{1}{An^n}\right)^{\tfrac 1{n-1}}- \beta_{n-1} \right)\vol(z_i(r_i))^{\frac{n}{n-1}}.
\fineq

On se demande alors s'il est loisible de choisir $A$ de sorte que (\ref{eq-A1}),  (\ref{eq-A2}) et  (\ref{eq-A3}) soient vraies avec la contrainte :
$$\left(\dfrac{1}{An^n}\right)^{\tfrac 1{n-1}}- \beta_{n-1} >0.$$
Mais on a :
$$
\left(\dfrac{1}{An^n}\right)^{\tfrac 1{n-1}}- \beta_{n-1} >0\Leftrightarrow  \dfrac{1}{An^n}> \beta_{n-1}^{n-1} 
 \Leftrightarrow  \dfrac{1}{\beta_{n-1}^{n-1} n^n}> A.
$$
\label{choixA}

Ainsi, on peut choisir $A$ tel que $\left(\dfrac{1}{An^n}\right)^{\tfrac 1{n-1}}- \beta_{n-1} >0$ de sorte que (\ref{eq-A1}),  (\ref{eq-A2}) et  (\ref{eq-A3}) soient vraies. Or, pour tout $i\ge i_0$, on a :
\eq
\lefteqn{\vol(c_i(r_i))-\volremp\big(z_i(r_i)\subset K(V))\big)\ge}\\&& 
\left(\left(\dfrac{1}{An^n}\right)^{\tfrac 1{n-1}}- \beta_{n-1} \right)\vol(z_i(r_i))^{\frac{n}{n-1}}\ge 0.
\fineq
Par sandwich, avec (\ref{2trefle}), on obtient $\vol(z_i(r_i))\tend 0$.

Selon (\ref{1trefle}), on obtient $\vol(C_i)\tend 0$ et $\volremp\big(z_i(r_i)\subset K(V)\big)\tend 0$. Selon  (\ref{2trefle}), on peut conclure que $\vol(c_i(r_i))\tend 0$. 
\end{proof}

\begin{lemm}
\label{lemmedoigts}
\textbf{Ou pourquoi il n'y a pas de doigt}. Quitte à extraire, la suite $(V'_i)$ converge pour la distance de Hausdorff vers une partie compacte propre $W_1$ de $W_\infty$.
\end{lemm}

\begin{proof}[Démonstration du lemme \ref{lemmedoigts}]
Quitte à extraire, la suite $(V'_i)$ converge pour la distance de Hausdorff vers une partie compacte $W_1$ de $K(V)$. On exhibe alors un point $w\in W_\infty$ qui n'est limite d'aucune suite $(w_i)$ telle que $w_i\in V_i'$ pour tout $i\ge i_0$. Rappelons que $c_i(r_i)=V_i\cap \tub\big(B(v_i,r_i),\rho\big)$, de sorte que, pour tout $i\ge i_0$, on a $\dist(v_i,z_i(r_i))\ge\rho$, où \og{}$\dist$\fg{} désigne la distance dans $K(V)$.

%%%%%%%%%%%%%%%%%%%%%%%%%%%%%%%%%%%%%%%%%%%%%%%%%%%
%%%%%%%%%%%%%%%%%%%%%%%%%%%%%%%%%%%%%%%%%%%%%%%%%%%
% FIGURE COUPER LES DOIGTS
%%%%%%%%%%%%%%%%%%%%%%%%%%%%%%%%%%%%%%%%%%%%%%%%%%%
%%%%%%%%%%%%%%%%%%%%%%%%%%%%%%%%%%%%%%%%%%%%%%%%%%%

\begin{figure}[h]
\begin{center}
\includegraphics[width=0.8\linewidth]{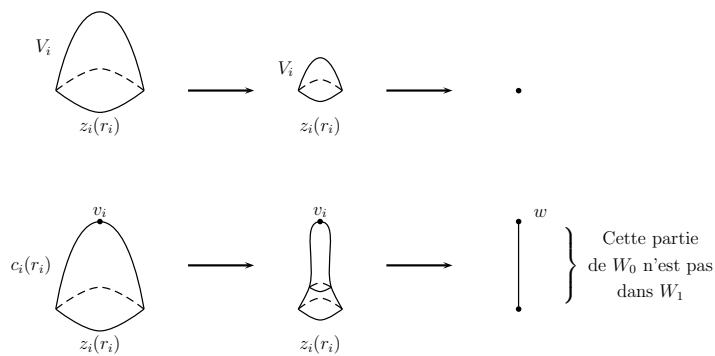}
\caption{Couper les doigts}
\end{center}
\end{figure}

%%%%%%%%%%%%%%%%%%%%%%%%%%%%%%%%%%%%%%%%%%%%%%%%%%%

Quitte à négliger une nouvelle extraction, on peut supposer que l'on a  : $v_i\tend w\in W_\infty$. 

Considérons une suite $(w_i)$ telle que, pour tout $i\ge i_0$, on ait $w_i\in C_i$. Rappelons nous aussi que $C_i$ est contenue dans le voisinage tubulaire de rayon $\eps_i=\beta_{n-1}\vol(z_i(r_i))^{\frac{1}{n-1}}$ de $z_i(r_i)$ et que l'on a $\eps_i\tend 0$. 

A partir d'un certain rang, on aura donc $\dist(w_i,w)\ge \tfrac {\rho} 2$, puisque l'on a $\dist(z_i(r_i),v_i)\ge\rho$. En effet :
$$\dist(z_i(r_i),v_i)\le\dist(z_i(r_i),w_i)+\dist(w_i,w)+\dist(w,v_i)$$
et $\dist(z_i(r_i),w_i)\tend 0$, $\dist(w,v_i)\tend 0$.\\
Le point $w$ ne peut être limite d'une suite $(w_i)$ telle que $w_i\in C_i$ pour tout $i\ge i_0$. Mais, pour tout $w'_i\in V_i'\setminus C_i$, on a $\dist(v_i,w'_i)\ge \rho$ : ainsi $w$ ne peut être limite d'une suite $(w_i)$ telle que $w_i\in V_i'$ à partir d'un certain rang. Selon le  résultat \textbf{R1} (page \pageref{H1}), $w\notin W_1$.

Montrons enfin que $W_1\subset W_\infty$. Soit $w_1\in W_1$. Il existe alors une suite $(v_i')$ dans $K(V)$ telle que $v_i'\tend w_1$ et $v_i'\in V_i'$ à partir d'un certain rang. Supposons, qu'à partir d'un certain rang, on ait $v_i'\notin C_i$. Le point $w_1$ est donc dans $W_\infty$, d'après le résultat \textbf{R1}. Sinon, quitte à extraire, on peut supposer que $v_i'\in C_i$ à partir d'un certain rang. Dans ce cas, $\dist(z_i(r_i),v_i')\tend 0$. Il existe donc une suite $(x_i)$ dans $K(V)$ telle que $x_i\in  z_i(r_i)$ et $\dist(x_i,v_i')\tend 0$. Par compacité, quitte à extraire, on peut supposer que $x_i\tend x$. D'après le  résultat \textbf{R1}, $x\in W_\infty$. Mais, par inégalité triangulaire, on obtient $v_i'\tend x$. Par unicité de la limite, il vient $w_1\in W_\infty$.

En conclusion, $W_1$ est une partie compacte propre de $W_\infty$ qui est limite d'une suite minimisante de pseudo-variétés, ce qui contredit le caractère minimal de $W_\infty$.  
\end{proof}

%%%%%%%%%%%%%%%%%%%%%%%%%%%%%%%%%%%%%%%%%%%%%%%%%%%%%%%%%%%%%%%%%%
%%%%%%%%%   MODIF
%%%%%%%%%%%%%%%%%%%%%%%%%%%%%%%%%%%%%%%%%%%%%%%%%%%%%%%%

Le lemme \ref{lemmevoisinagetub} est alors prouvé. Ce lemme étant vrai pour $\rho$ arbitrairement petit, le théorème B est alors démontré, la métrique riemannienne lisse par morceaux sur les pseudo-variétés $V_i$ étant celle fournie par la remarque \ref{rema-metrique} page \pageref{rema-metrique}. L'application $f_i$ étant égale à la restriction de $F$ à chaque $V_i$ (voir, pour les notations, le lemme \ref{prolongement} et la page \pageref{notations}).

%%%%%%%%%%%%%%%%%%%%%%%%%%%%%%%%%%%%%%%%%%%%%%%%%%%%%%%%%%%%%%%%%%%%%%%%%%%%%%%%%%%
\subsection{La démonstration du théorème A}
%%%%%%%%%%%%%%%%%%%%%%%%%%%%%%%%%%%%%%%%%%%%%%%%%%%%%%%%%%%%%%%%%%%%%%%%%%%%%%%%%%%
%%%%%%%%%%%%%%%%%%%%%%%%%%%%%%%%%%%%%%%%%%%%%%%%%%%%%%%%%%%%%%%%%%%%%%%%%%%%%%%%%%%
\begin{sloppypar}Fixons $\eps>0$. On prend un cycle géométrique normalisé $(V_1,f_1,g_1)$ tel que $\sys(V_1,f_1,g_1)=2$ et $\sigma(V_1,f_1,g_1)\le \sigma(h)+\eps_1$ où $\eps_1>0$ est petit. Pour $\eps_2>0$, suffisamment petit, et $V_0$ une partie suffisamment dense de $V_1$, on peut construire le complexe $K(V_1)$, ainsi que le plongement $V'=R_{\eps_2}\circ I_0(V_1)$ de $V_1$ dans $K(V_1)$. L'application $R_{\eps_2}\circ I_0$ étant lipschitzienne de rapport $\dfrac{1}{1-2\eps_2}$, on a  :\end{sloppypar}
$$\vol(V')\le \left(\dfrac{1}{1-2\eps_2}\right)^n\vol(V_1,g_1).$$

Puis, le théorème B nous donne, pour tout $a$ fixé dans $]0,1]$, une suite $(V_i,f_i,g_i)$ de cycles géométriques, qui représentent la classe d'homologie $h$, telle que :
\begin{quote}
\begin{enumerate}
	\item $\vol(V_i)\tend \vol[V']$ ;
	\item Pour $R\in [a,1]$, les boules $B(R)$ de rayon $R$ dans  chaque $V_i$ vérifient :
\begin{equation}
\label{volboule2}
		\vol (B(R))\ge A_n (R-a)^n
\end{equation}
\end{enumerate}
\end{quote}

Comme $\vol[V']\le  \vol(V',g_0)$, il vient $\vol[V']\le \left(\dfrac{1}{1-2\eps_2}\right)^n\vol(V_1,g_1)$. \\
Soit $\eps_3>0$. Pour $i$ suffisamment grand,  puisque $\vol(V_i)\tend \vol[V']$, on a :
\begin{equation}
\label{eq-inegavol}
	\vol (V_i)\le (1+\eps_3)\left(\dfrac{1}{1-2\eps_2}\right)^n\vol(V_1,g_1)
\end{equation}
Or $\vol(V_i,g_i)=\vol(V_i)$, et, d'après le théorème \ref{theosyst}, on a :$$2=\sys(V_1,f_1,f_1))\le \sys(V_i,f_i,g_i).$$ Il en résulte que :
\eq
\sigma(V_i,f_i,g_i)&\le &(1+\eps_3) \left(\dfrac{1}{1-2\eps_2}\right)^n\sigma(V_1,f_1,g_1)\\
&\le& (1+\eps_3)\left(\dfrac{1}{1-2\eps_2}\right)^n(\sigma(h)+\eps_1).
\fineq
Un choix judicieux de $\eps_1$, $\eps_2$ et $\eps_3$ permet d'obtenir, pour $i$ suffisamment grand  :
$$\sigma(V_i,f_i,g_i)\le \sigma(h)+\eps.$$
Puis on a $\sigma(h)\le \dfrac{\vol(V_i,g_i)}{\sys(V_i,f_i,g_i)^n}$, donc, en utilisant l'inégalité $\sigma(V_1,f_1,g_1)\le \sigma(h)+\eps_1$, on obtient : 
$$\sigma(V_1,f_1,g_1)-\eps_1\le \sigma(h)\le \sigma(V_i,f_i,g_i).$$
Cela amène :
$$\vol(V_1,f_1,g_1)-2^n\eps_1\le \dfrac{2^n}{\sys(V_i,f_i,g_i)^n}\vol(V_i,g_i).$$
A partir d'un certain rang, on obtient finalement, en utilisant l'inégalité (\ref{eq-inegavol}), $$\sys(V_i,f_i,g_i)^n\le  2^n \dfrac{\vol(V_i,g_i)}{\vol(V_1,f_1,g_1)-2^n\eps_1}\le 2^n(1+\eps_4),$$
pour tout $\eps_4>0$ fixé à l'avance. \\
%Il en résulte que, pour $i$ suffisamment grand et pour toute boule de rayon $R$ dans $[\eps,\tfrac 12 \sys(V_i,f_i,g_i)]$, l'inégalité (\ref{volboule2}) est vérifiée. 
Techniquement, on a donc prouvé que, pour tout $\eps>0$, $a>0$ et $b>0$ petits, il existe un cycle géométrique $(V,f,g)$ représentant $h$ tel que 
$$\sigma(V,f,g)\le \sigma(h)+\eps\et \vol(B(R))\ge A_n(R-a)^n$$
pour tout $R\in [a,\tfrac 12\sys(V,f,g)-b]$, ce qui achève la preuve du théorème A.

%%%%%%%%%%%%%%%%%%%%%%%%%%%%%%%%%%%%%%%%%%
%%%%%%%%%%%%%%%%%%%%%%%%%%%%%%%%%%%%%%%%%%
%%%%%%%%%%%%%%%%%%%%%%%%%%%%%%%%%%%%%%%%%%%%%%%%%%%%%%%%%%%%%%%%%%%%%%%%%%%%%%%%%%

%%%%%%%%%%%%%%%%%%%%%%%%%%%%%%%%%%%%%%%%%%%%%%%%%%%%%%%%%%%%%%%%%%%%%%%%%%%%%%%%%%
%%%%%%%%%%%%%%%%%%%%%%%%%%%%%%%%%%%%%%%%%%%%%%%%%%%%%%%%%%%%%%%%%%%%%%%%%%%%%%%%%%
\section{Autour du théorème A et de sa démonstration}
%%%%%%%%%%%%%%%%%%%%%%%%%%%%%%%%%%%%%%%%%%%%%%%%%%%%%%%%%%%%%%%%%%%%%%%%%%%%%%%%%%
%%%%%%%%%%%%%%%%%%%%%%%%%%%%%%%%%%%%%%%%%%%%%%%%%%%%%%%%%%%%%%%%%%%%%%%%%%%%%%%%%%
Je vais, dans ce paragraphe, donner quelques conséquences immédiates du théorème A et commenter quelques notions rencontrées, notamment donner une définition alternative des cycles réguliers.
%%%%%%%%%%%%%%%%%%%%%%%%%%%%%%%%%%%%%%%%%%%%%%%%%%%%%%%%%%%%%%%%%%%%%%%%%%%%%%%%%%
\subsection{Quelques conséquences du théorème A}
\begin{sloppypar}
La première conséquence concerne le volume systolique d'une classe d'homologie non nulle dans $H_n(\pi;\Z)$, lorsque $\pi$ est un groupe de présentation finie.
\end{sloppypar}
\begin{theo}
\label{theo-sigmah}
Soient $\pi$ un groupe de présentation finie et $n\ge 1$. Il existe une constante $C_n>0$, qui ne dépend que de $n$, telle que pour toute classe d'homologie non nulle $h$ dans $H_n(\pi;\Z)$ on ait :
$$\sigma(h)\ge C_n.$$ 
\end{theo}

\begin{proof}[Démonstration]
Soit $h\in H_n(\pi;\Z)$. Fixons  provisoirement $\eps>0$. Il existe alors, selon le théorème A, un cycle géométrique $\eps$-régulier $(V,f,g)$ qui représente la classe $h$. Soit $v\in V$. On a alors :
$$\vol(V,g)\ge \vol\big(B(v,\tfrac 12\sys(V,f,g)\big)\ge \dfrac{A_n}{2^n}\:\sys(V,f,g)^n.$$
Il en résulte que $\sigma (V,f,g)\ge \dfrac{A_n}{2^n}$. Ainsi :
$$\sigma(h)+\eps\ge \dfrac{A_n}{2^n}.$$
En faisant tendre $\eps$ vers $0$, on obtient le résultat souhaité, avec $C_n=\dfrac{A_n}{2^n}$.
\end{proof}

On peut alors en déduire une preuve de l'inégalité systolique de Gromov. 

\begin{gromov}Il existe une constante $C_n>0$ telle que pour toute variété  essentielle et orientable $M$ de dimension $n$ on ait :
$$\sigma(M)\ge C_n.$$
\end{gromov}

\begin{proof}[Démonstration]
Soit $M$ une variété essentielle orientable de dimension $n$, de groupe fondamental $\pi$. Il existe une application $f:M\to \K$, unique à homotopie près, telle que $f_*[M]=h\neq 0$, où 
$$f_*:H_n(M;\Z)\to H_n(\pi;\Z)$$
est le morphisme induit en homologie. D'après le théorème \ref{theo-sigmah}, on a $\sigma(h)\ge C_n$, pour une constante universelle $C_n>0$ qui ne dépend que de $n$. Mais $(M,f)$ est une représentation normalisée admissible de $h$ : selon \cite{Babenko-Balacheff-DVS}, on a $\sigma(M)=\sigma(h)$.
\end{proof}

%%%%%%%%%%%%%%%%%%%%%%%%%%%%%%%%%%%%%%%%%%%%%%%%%%%%%%%%%%%%%%%%%%%%%%%%%%%%%%%%%%
\subsection{Cycles réguliers et remplissage}
%%%%%%%%%%%%%%%%%%%%%%%%%%%%%%%%%%%%%%%%%%%%%%%%%%%%%%%%%%%%%%%%%%%%%%%%%%%%%%%%%%

Considérons un complexe $L$, muni d'une métrique riemannienne lisse par morceaux, dans lequel les cycles vérifient une inégalité isopérimétrique semblable à celle du théorème \ref{theo-inégalitéisocomplexes}. Plus précisément on suppose que la propriété suivante est vérifiée dans $L$.\\

\begin{inegiso}Pour tout cycle $z$ de dimension $n$ dans $L$, il existe une constante $c_n$ telle que :
\begin{equation}
\label{InegaliteisoL}
	\volremp(z\subset L)\le c_n[\vol(z)]^{\frac {n+1}n}.
\end{equation}
\end{inegiso}

\begin{defi}
Soit $(V,f,g)$ un cycle géométrique inclus $L$, $g$ étant la métrique induite par celle de $L$. Soit $\eps\in ]0,\tfrac12\sys(V,f,g)[$. On dira que $(V,f,g)$ est \textit{à remplissage $\eps$-régulier}\index{Cycle géométrique à remplissage régulier} lorsque pour tout $R$ dans 
$[\eps,\tfrac12\sys(V,f,g)[$, les boules $B_V(R)$ de rayon $R$ dans $V$ vérifient :
\begin{equation}
\label{inegalitecycreg}
	\vol (B_V(R))\le (1+\eps) \volremp (\partial B_V(R)\subset L)
\end{equation}
\end{defi}  

\begin{rema}
Comparer avec la définition donnée en \cite{Gromov-FRM}, \textbf{6.4} page 70.
\end{rema}

\begin{theo}
Soit $(V,f,g)$ un cycle géométrique à remplissage $\eps$-régulier dans $L$. Pour tout $R$ dans 
$[\eps,\tfrac12\sys(V,f,g)[$, Les boules $B_V(R)$ de rayon $R$ dans $V$ vérifient :
$$\vol (B_V(R))\ge C_n(R-\eps)^n,$$
pour une certaine constante $C_n$ qui ne dépend que de $n$.
\end{theo}

\begin{proof}[Démonstration]
Soit $R\in [\eps,\tfrac12\sys(V,f,g)[$. On pose $z(R)=\partial B_V(R)$, qui est un cycle de dimension $n-1$ dans $L$. Avec les inégalités (\ref{InegaliteisoL}) et (\ref{inegalitecycreg}), on obtient :
$$\vol (B_V(R))\le (1+\eps) C_n[\vol(\partial B_V(R))]^{\frac {n+1}n}.$$

Le lemme \ref{lem-analytique} assure alors que :
$$\vol\big(B_V(R)\big)\ge \dfrac1{(1+\eps)^{n-1}C_n^{n-1}(n)^{n}}\:(R-\eps)^n,$$
ce qui est le résultat souhaité avec $C_n=\dfrac1{(2c_n)^{n-1}(n)^{n}}$.
\end{proof}

\begin{rema}
Assurer l'existence de cycles à remplissage $\eps$-régulier dans le complexe $K(V)$ pourrait permettre de démontrer le théorème A. Mais une telle existence n'est pas plus facile à établir que la démarche proposée ici dans la preuve du théorème A.
\end{rema}

%%%%%%%%%%%%%%%%%%%%%%%%%%%%%%%%%%%%%%%%%%%%%%%%%%%%%%%%%%%%%%%%%%%%
%%%%%%%%%%%%%%%%%%%%%%%%%%%%%%%%%%%%%%%%%%%%%%%%%%%%%%%%%%%%%%%%%%%%
%\backmatter
%%%%%%%%%%%%%%%%%%%%%%%%%%%%%%%%%%%%%%%%%%%%%%%%%%%%%%%%%%%%%%%%%%%%
%%%%%%%%%%%%%%%%%%%%%%%%%%%%%%%%%%%%%%%%%%%%%%%%%%%%%%%%%%%%%%%%%%%%
\bibliographystyle{amsalpha}
\bibliography{biblio_cycles}
\end{document}